\documentclass[a4paper,12pt,reqno]{article}
\usepackage{amsthm}
\usepackage{amsmath}
\usepackage{amssymb}
\usepackage[mathscr]{eucal}
\usepackage[latin2]{inputenc}
\usepackage[T1]{fontenc}
\usepackage{indentfirst}
\usepackage{amsfonts}
\usepackage{graphicx}
\usepackage{color}
\usepackage{verbatim}

\newtheoremstyle{new_plain}
	{}
	{}
	{\itshape}
	{}
	{\sffamily\bfseries}
	{.}
	{5pt plus1pt minus1pt\relax}
	{\thmnumber{#2. }\thmname{#1}\thmnote{ #3}}

\newtheoremstyle{nev_circle_definition}
	{}
	{}
	{\normalfont}
	{}
	{\sffamily\bfseries}
	{.}
	{5pt plus1pt minus1pt\relax}
	{\thmnumber{#2. }\thmname{#1}\thmnote{ #3}}
	
\theoremstyle{plain}
	\newtheorem{theorem}{Theorem}
	\newtheorem{lemma}[theorem]{Lemma}
	\newtheorem{proposition}[theorem]{Proposition}

\theoremstyle{definition}

\theoremstyle{nev_circle_definition}

\numberwithin{theorem}{section}
\numberwithin{equation}{section}

\title{\bfseries \Large Universality limits for generalized Jacobi measures}
\author{\Large Tivadar Danka\thanks{This research was supported by ERC Advanced Grant No. 267055} \\[0.3cm] \large Bolyai Institute, University of Szeged \\ Aradi V. tere 1, H-6720 Szeged, Hungary \\ email: tdanka@math.u-szeged.hu}

\begin{document}

\maketitle

\begin{abstract}
In this paper universality limits are studied in connection with measures which exhibit power-type singular behavior somewhere in their support. We extend the results of Lubinsky for Jacobi measures supported on $ [-1,1] $ to generalized Jacobi measures supported on a compact subset of the real line, where the singularity can be located in the interior or at an endpoint of the support. The analysis is based upon the Riemann-Hilbert method, Christoffel functions, the polynomial inverse image method of Totik and the normal family approach of Lubinsky.
\end{abstract}


\section{Introduction}

Universality limits for random matrices is an intensively studied topic of mathematics and mathematical physics, which has several applications even outside physics and mathematics. For ensembles of $ n \times n $ Hermitian random matrices invariant under unitary conjugation, a connection with orthogonal polynomials can be established. If the eigenvalue distribution is given by 
\[
p(x_1, \dots, x_n) = \frac{1}{Z_n} \prod_{1 \leq i < j \leq n} |x_j - x_i|^2 \prod_{k=1}^{N} w(x_k) dx_k,
\]
then the $ k $-point correlation functions defined by
\[
	R_{k,n}(x_1, \dots, x_k) = \frac{n!}{(n-k)!} \int \dots \int p(x_1, \dots, x_n) dx_{k + 1} \dots dx_n
\]
can be expressed as
\begin{equation}\label{correlation_function_CD_kernel}
	R_{k,n}(x_1, \dots, x_k) = \det\Big(\sqrt{w(x_i)w(x_j)}K_{n}(x_i,x_j)\Big)_{i,j=1}^{n},
\end{equation}
where, if $ p_k(\mu,x) = p_k(x) $ denotes the orthonormal polynomial of degree $ k $ with respect to the measure $ d\mu(x) = w(x)dx$, the function $ K_n(x,y) $ is the so-called Christoffel-Darboux kernel for $ \mu $ defined as
\begin{equation}\label{Christoffel-Darboux_kernel}
	K_n(x,y) = \sum_{k=0}^{n-1} p_k(x) p_k(y).
\end{equation}
This was originally shown for Gaussian ensembles by Mehta and Gaudin in \cite{Mehta-Gaudin}, but later this technique was developed for more general ensembles, see in particular \cite[(4.89)]{Deift-Gioev} or for example \cite{Deift} \cite{Deift_1} \cite{Deift_2} \cite{Pastur-Shcherbina}. $ K_n(x,y) $ can be expressed in terms of $ p_{n} $ and $ p_{n-1} $ as
\begin{equation}\label{Christoffel-Darboux}
	K_n(x,y) = \frac{\gamma_{n-1}}{\gamma_{n}} \frac{p_{n}(x) p_{n-1}(y) - p_{n-1}(x)p_{n}(y)}{x - y},
\end{equation}
where $ \gamma_n(\mu) = \gamma_n $ denotes the leading coefficient of $ p_n(x) $. This is called the Christoffel-Darboux formula.

Because of (\ref{correlation_function_CD_kernel}), limits of the type
\begin{equation}\label{universality_limit}
	\lim_{n \to \infty} \frac{K_n\big(x_0 + \frac{a}{n}, x_0 + \frac{b}{n} \big)}{K_n(x_0, x_0)}, \quad a,b \in \mathbb{R},
\end{equation}
which are called universality limits, are playing an especially important role in the study of eigenvalue distributions for random matrices. For measures supported on $ [-1,1] $, a new approach for universality limits was developed by D. S. Lubinsky in the seminal papers \cite{Lubinsky_1} \cite{Lubinsky_2} \cite{Lubinsky_3}. In \cite{Lubinsky_1} it was shown that if $ \mu $ is a finite Borel measure supported on $ [-1,1] $ which is regular in the sense of Stahl-Totik (see (\ref{reg}) below) and absolutely continuous with $ d\mu(x) = w(x)dx $ in a neighbourhood of $ x_0 \in (-1,1) $, where $ w(x)$ is also continuous and strictly positive,
\begin{equation}\label{Lubinsky_bulk}
	\lim_{n \to \infty} \frac{\widetilde{K}_n \Big( x_0 + \frac{a}{\widetilde{K}_n(x_0, x_0)}, x_0 + \frac{b}{\widetilde{K}_n(x_0, x_0)} \Big)}{\widetilde{K}_n(x_0, x_0)} = \frac{\sin \pi (b - a)}{\pi (b - a)}
\end{equation}
holds, where $ \widetilde{K}_n(x,y) = \sqrt{w(x)w(y)}K_n(x,y) $ denotes the so-called normalized Christoffel-Darboux kernel. In the previous results analiticity of the weight function was required on the whole support $ [-1,1] $, therefore this was a large step ahead.

An important part of Lubinsky's method is that if one is able to deduce limits of the type (\ref{universality_limit}) with $ a = b $, then this can be used to obtain (\ref{universality_limit}) in general. The analysis was largely based upon Christoffel functions, whose behavior is well known for a relatively large class of measures. For a finite Borel measure $ \mu $ the $ n $-th Christoffel function is defined as
\begin{equation}\label{Christoffel_definition}
	\lambda_n(\mu,x_0) = \inf_{\deg (P_n) < n} \int \frac{|P_n(x)|^2}{|P_n(x_0)|^2} d\mu(x),
\end{equation}
where the infimum is taken for all polynomials of degree at most $ n-1 $ with $ P_n(x_0) \neq 0 $. (In other words, it is the $(-1/2)$-th power of the norm of the evaluation functional at $ x_0 $ defined on $ \mathcal{P}_{n-1} \cap L^2(\mu) $, where $ \mathcal{P}_n $ denotes the subspace of polynomials of degree at most $ n $.) The Christoffel functions can be expressed in terms of the Christoffel-Darboux kernel as
\begin{equation}\label{Christoffel_polynomial_form}
	\lambda_n(\mu, x_0) = \frac{1}{K_{n}(x_0,x_0)},
\end{equation}
hence
\[
	\frac{K_n\big(x_0 + \frac{a}{n}, x_0 + \frac{a}{n} \big)}{K_n(x_0, x_0)} = \frac{\lambda_n(\mu, x_0)}{\lambda_n \big(\mu, x_0 + \frac{a}{n} \big)},
\]
holds, therefore this way universality limits can be translated to Christoffel functions. This has proven to be very useful, because Christoffel functions exhibit strong localization properties. Asymptotics for Christoffel functions has been studied since the beginning of the 20th century. The earliest results can be originated from Szeg\H{o}, who studied measures supported on the unit circle. For the classical results  and a detailed account see \cite{Freud} and \cite{Nevai_1}. Szeg\H{o}'s early result was extended by A. M\'at\'e, P. Nevai and V. Totik in the seminal paper \cite{Mate-Nevai-Totik}, which considered measures supported on the unit circle and on the interval $ [-1,1] $. For measures supported on arbitrary compact subsets of the real line asymptotics was established by V. Totik in \cite{Totik_1} using the polynomial inverse image method, which was developed by him in \cite{Totik_2}. (The roots of this method can be originated from the paper \cite{Geronimo-vanAssche}. For more details on this method, see the aforementioned article or the survey paper \cite{Totik_4}.) Lubinsky's result was simultaneously extended for measures supported on arbitrary compact subsets of the real line by B. Simon in \cite{Simon_1} and by V. Totik in \cite{Totik_3}, although they used very different methods. \\

When the measure exhibits singular behavior at the prescribed point $ x_0 $, for example it behaves like $ |x - x_0|^\alpha dx $ for some $ \alpha > -1 $, it no longer shows the same behavior and instead of the sinc kernel, something else appears. Measures of the form
\[
	d\mu(x) = (1 - x)^\alpha (1 + x)^\beta h(x) dx, \quad x \in [-1,1],
\]
where $ h(x) $ is positive and analytic, were studied by A. B. J. Kuijlaars and M. Vanlessen in \cite{Kuijlaars-Vanlessen}. Using Riemann-Hilbert methods, they showed that
\begin{equation}\label{universality_limit_jacobi_edge}
	\lim_{n \to \infty} \frac{1}{2n^2} \widetilde{K}_n\Big( 1 - \frac{a}{2n^2}, 1 - \frac{b}{2n^2} \Big) = \mathbb{J}_{\alpha}(a,b)
\end{equation}
uniformly for $ a, b $ in compact subsets of $ (0, \infty) $, where $ \mathbb{J}_\alpha(a,b) $ is the so-called Bessel kernel defined as
\begin{equation}\label{bessel_kernel}
	\mathbb{J}_\alpha(a,b) = \frac{J_\alpha(\sqrt{a}) \sqrt{b} J_\alpha^{\prime}(\sqrt{b}) - J_\alpha(\sqrt{b}) \sqrt{a} J_\alpha^{\prime}(\sqrt{a})}{2(a - b)}
\end{equation}
and $ J_\alpha(x) $ denotes the Bessel function of the first kind and order $ \alpha $. (Actually, they showed a much stronger result, from which (\ref{universality_limit_jacobi_edge}) follows.) This was extended by Lubinsky in \cite{Lubinsky_2}. He proved that if $ \mu $ is a finite Borel measure supported on $ [-1,1] $ which is absolutely continuous on $ [1-\varepsilon, 1] $ for some $ \varepsilon > 0 $ with
\[
	d\mu(x) = w(x) |x-1|^\alpha, \quad x \in [1-\varepsilon, 1]
\]
there, where $ w(x) $ is strictly positive and continuous at $ 1 $, then
\[
	\frac{1}{2n^{2\alpha + 2}} K_n\Big( 1 - \frac{a}{2n^2}, 1 - \frac{b}{2n^2} \Big) = \mathbb{J}_\alpha^{*}(a,b)
\]
holds, where $ \mathbb{J}_\alpha^*(a,b) = \frac{\mathbb{J}_\alpha(a,b)} {a^{\alpha/2} b^{\alpha/2}} $. It was also shown by Lubinsky in \cite{Lubinsky_3} that if $ K $ is a compact subset of the real line and $ x_0 \in K $ is a right endpoint of $ K $ (i.e. there exists an $ \varepsilon > 0 $ such that $ K \cap (x_0, x_0 + \varepsilon) = \varnothing $), then if $ \mu $ is a finite Borel measure with  $ \operatorname{supp}(\mu) = K $ which is absolutely continuous in a small left neighbourhood of $ x_0 $ with
\[
	d\mu(x) = |x - x_0|^\alpha dx
\]
there for some $ \alpha > -1 $, then
\begin{equation}\label{Lubinsky_diagonal_limit}
	\lim_{n \to \infty} \frac{K_n(x_0 - a \eta_n, x_0 - a \eta_n)}{K_n(x_0,x_0)} = \frac{\mathbb{J}_\alpha^{*}(a, a)}{\mathbb{J}_\alpha^{*}(0, 0)}
\end{equation}
for all $ a \in [0, \infty) $  implies
\begin{equation}\label{Lubinsky_hard_edge_universality_limit}
	\lim_{n \to \infty} \frac{K_n(x_0 - a \eta_n, x_0 - b \eta_n)}{K_n(x_0,x_0)} = \frac{\mathbb{J}_\alpha^{*}(a, b)}{\mathbb{J}_\alpha^{*}(0, 0)}
\end{equation}
uniformly for $ a, b $ on compact subsets of the complex plane, where the sequence $ \eta_n $ is $ \eta_n = (\mathbb{J}_\alpha^*(0,0)/K_n(x_0, x_0))^{1/(\alpha + 1)} $. In such a general setting, it was not known if (\ref{Lubinsky_diagonal_limit}) holds. 

The aim of this paper is twofold. On one hand, we will show that (\ref{Lubinsky_diagonal_limit}) does indeed hold, hence (\ref{Lubinsky_hard_edge_universality_limit}) also holds as well. On the other hand, we also aim to establish universality limits in the case when the singularity is located in the interior of the support rather than on the hard edge. \\

To study universality limits for measures supported on general compact sets, we shall need a few concepts from logarithmic potential theory, most importantly the concept of equilibrium measures. For a detailed account on logarithmic potential theory, see the books \cite{Ransford} and \cite{Saff-Totik}. If $ \mu $ is a finite Borel measure supported on the complex plane, its \emph{energy} is defined as
\[
	I(\mu) = \int \int \log \frac{1}{|z - w|} d\mu(z) d\mu(w).
\]
We can define the energy of a \emph{set} $ K \subseteq \mathbb{C} $ as the infimum of energies for probability measures supported inside $ K $, i.e.
\[
	I(K) = \inf_{\mu \in \mathcal{M}_1(K)} I(\mu),
\]
where $ \mathcal{M}_1(K) $ denotes the set of Borel probability measures with support lying in $ K $. (The quantity $ I(K) $ is also known as \emph{Robin's constant}.) The logarithmic capacity of $ K $ is defined as
\[
	\operatorname{cap}(K) = e^{-I(K)}.
\]
Sets of zero logarithmic capacity are called \emph{polar sets}. They are playing the role of negligible sets in logarithmic potential theory. If $ K $ is a compact subset of the complex plane with nonzero logarithmic capacity, then there is a unique measure denoted by $ \nu_K $ such that $ I(\nu_K) = I(K) $. The measure $ \nu_K $ is called the \emph{equilibrium measure} for $ K $, and its Radon-Nikodym derivative, if it exists, is denoted by $ \omega_K(x) $. For example, we have
\begin{equation}\label{equilibrium_density_unit_interval}
	\omega_{[-1,1]}(x) = \frac{1}{\pi \sqrt{1 - x^2}},
\end{equation}
which is the arcsine distribution.

For a domain $ D \subseteq \mathbb{C}_\infty $ which contains a neighbourhood of $ \infty $, the \emph{Green's function} with pole at infinity is the unique function $ g_D(\cdot, \infty): D \to [-\infty, \infty) $ for which \\
(a) $ g_D(z, \infty) $ is harmonic on $ D \subseteq \mathbb{C}_\infty $ and bounded outside the neighbourhoods of $ \infty $, \\
(b) $ g_D(z, \infty) = \log |z| + O(1) $ as $ z \to \infty $, \\
(c) $ g_D(z, \infty) \to 0 $ as $ z \to \xi \in \partial D \setminus H $, where $ H $ is a set of zero logarithmic capacity.

A compact set $ K $, if $ \Omega $ denotes the unbounded component of its complement, is said to be \emph{regular with respect to the Dirichlet problem}, if $ g_\Omega(z, \infty) \to 0 $ as $ z \to \xi $ for all $ \xi \in \partial \Omega $, i.e. the exceptional set $ H $ in property (c) is empty. \\

Along with local conditions imposed on the measure, (for example the Szeg\H{o} condition like in \cite{Mate-Nevai-Totik}, continuity of weight function like in \cite{Lubinsky_1}, or singular behavior of type $ |x-x_0|^\alpha dx $ as in our case) some kind of global condition is needed. The so-called Stahl-Totik regularity is such a property. A measure $ \mu $ is said to be regular in the sense of Stahl-Totik (or $ \mu \in \mathbf{Reg} $ in short), if for every sequence of nonzero polynomials $ \{ P_n \}_{n=0}^{\infty} $,
\begin{equation}\label{reg}
	\limsup_{n \to \infty} \bigg( \frac{|P_n(z)|}{\| P_n \|_{L^2(\mu)}} \bigg)^{1/\deg(P_n)} \leq 1
\end{equation}
holds for all $ z \in \operatorname{supp}(\mu) \setminus H $, where $ H $ is a set of zero logarithmic capacity. If $ \mathbb{C} \setminus \operatorname{supp}(\mu) $ is regular with respect to the Dirichlet problem, Stahl-Totik regularity is equivalent with the uniform estimate
\begin{equation}\label{reg_dirichlet}
	\limsup_{n \to \infty} \bigg( \frac{\| P_n \|_{\operatorname{supp}(\mu)}}{\| P_n \|_{L^2(\mu)}} \bigg)^{1/\deg(P_n)} \leq 1.
\end{equation}
There are several criteria for regularity, most notably the Erd\H{o}s-Tur\'an criterion. In a special case, it says that if $ \mu $ is a measure supported on the interval $ [-1,1] $ and it is absolutely continuous with $ d\mu(x) = w(x)dx $, then $ "w(x) > 0 \textit{ almost everywhere on } [-1,1] " $ implies that $ \mu $ is regular in the sense of Stahl-Totik. For a detailed account on the $ \mathbf{Reg} $ class for measures, see \cite{Stahl-Totik}. \\

In order to express universality limits for measures exhibiting power-type singularity in the interior of its support, we define the kernel function for $ a, b \in \mathbb{R} $ by
\begin{equation}\label{kernel_function}
	\mathbb{L}_\alpha(a,b) =
	\begin{cases}
		\frac{\sqrt{ab}}{2(a - b)} \Big( J_{\frac{\alpha + 1}{2}}(a) J_{\frac{\alpha - 1}{2}}(b) - J_{\frac{\alpha + 1}{2}}(b) J_{\frac{\alpha - 1}{2}}(a) \Big) & \text{if } a, b \geq 0,\\
		\frac{\sqrt{a(-b)}}{2(a - b)} \Big( J_{\frac{\alpha + 1}{2}}(a) J_{\frac{\alpha - 1}{2}}(-b) + J_{\frac{\alpha + 1}{2}}(-b) J_{\frac{\alpha - 1}{2}}(a) \Big) & \text{if } a \geq 0, b < 0,\\
		\mathbb{L}_\alpha(-a,-b) & \text{else},
	\end{cases}
\end{equation}
where $ J_\nu(x) $ denotes the Bessel functions of the first kind and order $ \nu $. Note that
\[
	\mathbb{L}_\alpha(a,a) = \frac{|a|}{2} \Big( J_{\frac{\alpha+1}{2}}^{\prime}(|a|) J_{\frac{\alpha-1}{2}}(|a|) - J_{\frac{\alpha+1}{2}}(|a|) J_{\frac{\alpha-1}{2}}^{\prime}(|a|) \Big).
\]
Since $ J_\nu(z) = z^\nu G(z) $ where $ G(z) $ is an entire function, we can define the entire version of the kernel function for arbitrary complex arguments as
\begin{equation}\label{kernel_star_function}
	\mathbb{L}_{\alpha}^{*}(a,b) = \frac{\mathbb{L}_\alpha(a,b)}{a^{\alpha/2} b^{\alpha/2}}, \quad \mathbb{L}_{\alpha}^{*}(a) = \frac{\mathbb{L}_\alpha(a,a)}{a^{\alpha}}, \quad a, b \in \mathbb{C}.
\end{equation}
 We emphasize that $ \mathbb{L}_\alpha(a,b) $ is defined with different formulas for the cases $ ab \geq 0 $ and $ ab < 0 $, and without the normalization in the definition (\ref{kernel_star_function}), this would cause problems. This way although, using that $ J_\nu(-z) = (-1)^\nu J_\nu(z) $, we see that in fact the two formulas in (\ref{kernel_function}) coincide after normalization. \\

Our aim is to prove the following four theorems. The first two deals with the asymptotics of Christoffel functions when the power type singularity is in the interior (in other words, in the bulk) or at an endpoint (in other words, at the hard edge). The last two theorems are concerned with universality limits in the same cases.
\begin{theorem}\label{main_theorem_bulk}
Let $ \mu $ be a finite Borel measure which is regular in the sense of Stahl-Totik and suppose that $ \mu $ is supported on a compact set $ K = \operatorname{supp}(\mu) $ on the real line. Let $ x_0 \in \operatorname{int}(K) $ be a point from the interior of its support and suppose that on some interval $ (x_0 - \varepsilon_0, x_0 + \varepsilon_0) $ containing $ x_0 $, the measure $ \mu $ is absolutely continuous with
\[
	d\mu(x) = w(x) |x - x_0|^\alpha dx, \quad x \in (x_0 - \varepsilon_0, x_0 + \varepsilon_0)
\]
there for some $ \alpha > -1 $ and $ \alpha \neq 0 $, where $ w $ is strictly positive and continuous at $ x_0 $. Then
\begin{equation}\label{general_bulk}
	\lim_{n \to \infty} n^{\alpha + 1} \lambda_n \Big(\mu, x_0 + \frac{a}{n} \Big) = \frac{w(x_0)}{(\pi \omega_K(x_0))^{\alpha + 1}} \Big( \mathbb{L}_{\alpha}^{*}\big(\pi \omega_K(x_0) a \big) \Big)^{-1}
\end{equation}
holds uniformly for $ a $ in compact subsets of the real line, where $ \mathbb{L}_\alpha^*(\cdot) $ is defined by (\ref{kernel_star_function}).
\end{theorem}

The analogue for the edge is the following.

\begin{theorem}\label{main_theorem_edge}
Let $ \mu $ be a finite Borel measure which is regular in the sense of Stahl-Totik and suppose that $ \mu $ is supported on a compact set $ K = \operatorname{supp}(\mu) $ on the real line. Let $ x_0 \in K $ be a right endpoint of $ K $ (i.e. $ K \cap (x_0, x_0 + \varepsilon_1) = \varnothing $ for some $ \varepsilon_1 > 0 $) and assume that on some interval  $ (x_0 - \varepsilon_0, x_0] $ the measure $ \mu $ is absolutely continuous with
\[
	d\mu(x) = w(x)|x - x_0|^\alpha dx, \quad x \in (x_0 - \varepsilon_0, x_0] \\
\]
there for some $ \alpha > -1 $, where $ w $ is strictly positive and left continuous at $ x_0 $. Then
\begin{equation}\label{general_edge}
	\lim_{n \to \infty} n^{2\alpha + 2} \lambda_n \Big( \mu, x_0 - \frac{a}{2n^2} \Big) = \frac{w(x_0)}{M(K, x_0)^{2\alpha + 2}} \Big( 2^{\alpha + 1} \mathbb{J}_\alpha^*\big(M(K, x_0)^2 a \big) \Big)^{-1}
\end{equation}
holds uniformly for $ a $ in compact subsets of $ [0, \infty) $, where $ \mathbb{J}_\alpha^*(\cdot) $ is the Bessel kernel defined by (\ref{bessel_kernel}) and $ M(K, x_0) $ is defined by
\begin{equation}\label{M_equilibrium}
	M(K, x_0) = \lim_{x \to x_0 -} \sqrt{2} \pi |x - x_0|^{1/2} \omega_K(x).
\end{equation}
\end{theorem}

By symmetry, there is a similar result for left endpoints. Both of these theorems are in agreement with Theorems 1.1 and 1.2 of \cite{Danka-Totik} in the case when $ K $ is a finite union of intervals and $ a = 0 $. From the asymptotics for Christoffel functions we obtain universality limits.

\begin{theorem}\label{main_theorem_universality_bulk}
With the assumptions of Theorem \ref{main_theorem_bulk}, we have
\begin{equation}\label{general_universality_bulk}
	\lim_{n \to \infty} \frac{K_n\big(x_0 + \frac{a}{n}, x_0 + \frac{b}{n} \big)}{K_n(x_0, x_0)} = \frac{\mathbb{L}_\alpha^*(\pi \omega_{K}(x_0) a, \pi \omega_{K}(x_0) b)}{\mathbb{L}_\alpha^*(0,0)}
\end{equation}
uniformly for $ a,b $ in compact subsets of the complex plane.
\end{theorem}

\begin{theorem}\label{main_theorem_universality_edge}
With the assumptions of Theorem \ref{main_theorem_edge}, we have
\begin{equation}\label{general_universality_edge}
	\lim_{n \to \infty} \frac{K_n\big( x_0 - \frac{a}{2n^2}, x_0 - \frac{b}{2n^2} \big)}{K_n(x_0, x_0)} = \frac{\mathbb{J}_\alpha^*\big(M(K, x_0)^2 a, M(K, x_0)^2 b\big)}{\mathbb{J}_\alpha^*(0,0)}.
\end{equation}
uniformly for $ a, b $ in compact subsets of the complex plane.
\end{theorem}
Again by symmetry, there is a similar result for left endpoints. The proof of Theorems \ref{main_theorem_bulk} and \ref{main_theorem_edge} consists of two steps. First we prove (\ref{general_bulk}) and (\ref{general_edge}) for a special case, namely we study the measures $ \mu_\alpha^b $ and $ \mu_\alpha^e $ defined as
\[
	d\mu_\alpha^b(x) = |x|^\alpha dx, \quad x \in [-1,1]
\]
and
\[
	d\mu_\alpha^e(x) = |x-1|^\alpha dx, \quad x \in [-1,1].
\]
Although (\ref{general_edge}) for $ \mu_\alpha^e $ was previously known, one has to go a longer way to obtain (\ref{general_bulk}) for $ \mu_\alpha^b $. This is done in Section \ref{section_model_cases} with the use of Riemann-Hilbert methods. In Section \ref{section_Christoffel_functions} we use the polynomial inverse image method of Totik to prove Theorems \ref{main_theorem_bulk} and \ref{main_theorem_edge} in their full generality. After this, Theorem \ref{main_theorem_universality_edge} follows immediately from Theorem \ref{main_theorem_edge} by applying Lubinsky's result \cite[Theorem 1.2]{Lubinsky_3}. To prove Theorem \ref{main_theorem_universality_bulk}, we have to work some more, since Lubinsky's theorem cannot be applied to points in the interior of the support. Therefore we have to build the same machinery to handle this case, which will be done in Section \ref{section_universality}.

\subsection{Acknowledgements} I would like to express my deepest gratitude to Doron Lubinsky and Vilmos Totik for the useful discussions throughout the preparation of this paper.

\section{Model cases}\label{section_model_cases}

In this section our goal is to prove (\ref{general_universality_bulk}) and (\ref{general_universality_edge}) for the measures $ \mu_\alpha^b $ and $ \mu_\alpha^e $ which are supported on $ [-1,1] $ and given by
\begin{equation}\label{mu_alpha_bulk}
	d\mu_\alpha^b(x) = |x|^{\alpha} dx, \quad x \in [-1,1],
\end{equation}
and
\begin{equation}\label{mu_alpha_edge}
	d\mu_\alpha^e(x) = |x-1|^{\alpha} dx, \quad x \in [-1,1],
\end{equation}
where $ \alpha > -1 $ is arbitrary.

\subsection{Model case for the bulk}

First we prove (\ref{general_universality_bulk}) for $ \mu_\alpha^b $. Although $ \lambda_n(\mu_\alpha^b, 0) $ was studied in \cite{Nevai_2} (along with $ \lambda_n(\mu_\alpha^e, 1) $), we need to study $ \lambda_n(\mu_\alpha^b, a/n) $ for an arbitrary $ a $. To do this, we use the Riemann-Hilbert method. During this part we follow closely in the lines of \cite{Kuijlaars-Vanlessen} and \cite{Vanlessen}. Although the Riemann-Hilbert analysis for generalized Jacobi measures was carried out by M. Vanlessen in \cite{Vanlessen}, it does not cover the asymptotics for the Christoffel-Darboux kernels. First we define a Riemann-Hilbert problem for the $ 2 \times 2 $ matrix-valued function $ Y(z): \mathbb{C} \to \mathbb{C}^{2 \times 2} $, which can be expressed in terms of the orthogonal polynomials. Then via a series of transformations $ Y \mapsto T \mapsto S \mapsto R $ a $ 2 \times 2 $ matrix-valued function $ R(z) $ can be obtained for which asymptotic behavior is known. These transformations can be unraveled to obtain strong asymptotics for the orthogonal polynomials for $ \mu_\alpha^b $ which will yield (\ref{general_universality_bulk}) in this special case. Since this section does not contain new ideas, we shall carry out the analysis as briefly as possible. The interested reader can find the details in \cite{Kuijlaars-Vanlessen} and \cite{Vanlessen}.

\begin{proposition}\label{model_case_bulk_Christoffel_proposition}
Let $ \mu_\alpha^b $ be the measure supported on $ [-1,1] $ defined as
\[
	d\mu_\alpha^b(x) = |x|^\alpha dx, \quad x \in [-1,1],
\]
where $ \alpha > -1 $ and $ \alpha \neq 0 $. Then for the normalized Christoffel-Darboux kernel,
\begin{equation}\label{model_case_bulk_normalized}
	\frac{1}{n} \widetilde{K}_n\bigg(\frac{a}{n}, \frac{b}{n} \bigg) = \mathbb{L}_\alpha(a, b) + O\bigg(\frac{|a|^{\alpha/2} |b|^{\alpha/2}}{n} \bigg)
\end{equation}
holds uniformly for $ a, b $ in bounded subsets of $ (-\infty, 0) \cup (0, \infty) $, where $ \mathbb{L}_\alpha(a,b) $ is defined by (\ref{kernel_function}). Moreover, for the non-normalized Christoffel-Darboux kernel, we have
\begin{equation}\label{model_case_bulk}
	\frac{1}{n^{\alpha + 1}}K_n \bigg( \frac{a}{n}, \frac{b}{n} \bigg) = \mathbb{L}_{\alpha}^{*}(a,b) + O(1/n)
\end{equation}
uniformly for $ a, b $ in compact subsets of the real line, where $ \mathbb{L}_\alpha^*(a) $ is defined by (\ref{kernel_star_function}).
\end{proposition}
\begin{proof}
First we show (\ref{model_case_bulk_normalized}) using the Riemann-Hilbert method following in the steps of Kuijlaars and Vanlessen \cite{Kuijlaars-Vanlessen} and Vanlessen \cite{Vanlessen}, then we show (\ref{model_case_bulk}) by normalizing and appealing to known results for $ a = b = 0 $. We define the Riemann-Hilbert problem for the $ 2 \times 2 $ matrix valued function $ Y(z) = (Y_{ij}(z))_{i,j=1}^{2} $ as in \cite{Vanlessen}. Suppose that \\
(a) $ Y(z) $ is analytic for $ z \in \mathbb{C} \setminus [-1,1] $. \\[0.05cm]
(b) For all $ x \in (-1,0) \cup (0,1) $ the limits
\[
	Y_+(x) = \lim_{\substack{z \to x \\ \operatorname{Im}(z) > 0}} Y(z), \quad Y_-(x) = \lim_{\substack{z \to x \\ \operatorname{Im}(z) < 0}} Y(z)
\]
exist and the jump condition
\[
	Y_+(x) = Y_-(x) \begin{pmatrix} 1 & |x|^\alpha \\ 0 & 1 \end{pmatrix}, \quad x \in (-1,0) \cup (0,1)
\]
holds. \\[0.05cm]
(c) For the behavior of $ Y(z) $ near infinity we have
\[
	Y(z) = (I + O(z^{-1})) \begin{pmatrix} z^n & 0 \\ 0 & z^{-n} \end{pmatrix},
\]
as $ z \to \infty $. \\[0.05cm]
(d) For the behavior of $ Y(z) $ near $ 0 $ we have
\[
	Y(z) = \begin{cases}
	O \begin{pmatrix} 1 & |z|^{\alpha/2} \\ 1 & |z|^{\alpha/2} \end{pmatrix} & \textnormal{if } \alpha < 0  \\
	O \begin{pmatrix} 1 & 1 \\ 1 & 1 \end{pmatrix} & \textnormal{if } \alpha > 0
	\end{cases}
\]
as $ z \to 0 $ in $ \mathbb{C} \setminus [-1,1] $. \\[0.05cm]
(e) For the behavior of $ Y(z) $ near the endpoints $ (-1)^k $ for $ k = 1, 2 $ we have
\[
	Y(z) = \begin{pmatrix} 1 & \log |z - (-1)^k| \\ 1 & \log |z - (-1)^k | \end{pmatrix}
\]
as $ z \to (-1)^k $ in $ \mathbb{C} \setminus [-1,1] $. \\

The unique solution for this Riemann-Hilbert problem can be expressed in terms of orthogonal polynomials. If $ \pi_n(\mu_\alpha^b, z) = \pi_n(z) $ denotes the monic orthogonal polynomial of degree $ n $ with respect to the measure $ \mu_\alpha^b $ and $ \gamma_n(\mu_\alpha^b) = \gamma_n $ denotes the leading coefficient of the orthonormal polynomial $ p_n(\mu_\alpha^b, z) = p_n(z) $, then, see \cite[Theorem 2.2]{Vanlessen}, $ Y(z) $ takes the form
\begin{equation}\label{model_case_Y_OP}
	Y(z) =
	\begin{pmatrix}
	\pi_n(z) & \frac{1}{2\pi i} \int_{-1}^{1} \frac{\pi_n(x)}{x - z} d\mu_\alpha^b(x) \\
	-2\pi i \gamma_{n-1}^{2} \pi_{n - 1}(z) & - \gamma_{n-1}^{2} \int_{-1}^{1} \frac{\pi_{n - 1}(x)}{x - z} d\mu_\alpha^b(x)
	\end{pmatrix}.
\end{equation}
To give an asymptotic formula for $ Y(z) $, we need to introduce some special functions. Let $ w(z) $ be an analytic continuation of the function $ |x|^\alpha $ to the two half-planes defined by
\begin{equation}\label{model_case_definition_w}
	w(z) = 
	\begin{cases}
		(-z)^{\alpha} & \textnormal{if } \operatorname{Re}(z) < 0, \\
		z^{\alpha} & \textnormal{if } \operatorname{Re}(z) > 0.
	\end{cases}
\end{equation}
as in \cite[(3.4)]{Vanlessen}, and define the function $ W(z) $ for all $ z \in \mathbb{C} \setminus \mathbb{R} $ similarly by
\begin{equation}\label{model_case_definition_W}
	W(z) =
	\begin{cases}
		z^{\alpha/2} & \textnormal{if } \operatorname{Re}(z) < 0, \\
		(-z)^{\alpha/2} & \textnormal{if } \operatorname{Re}(z) > 0
	\end{cases}
\end{equation}
so that overall, we have
\begin{equation}\label{W_w_relation}
	W^2(z) =
	\begin{cases}
		w(z) e^{-i \pi \alpha} & \text{if } \operatorname{Re}(z) \operatorname{Im}(z) \geq 0, \\
		w(z) e^{i \pi \alpha} & \text{if } \operatorname{Re}(z) \operatorname{Im}(z) < 0.
	\end{cases}
\end{equation}
The function $ \varphi(z) = z + \sqrt{z^2 - 1} $ denotes the conformal mapping of $ \mathbb{C} \setminus [-1,1] $ onto the exterior of the unit circle and the auxiliary function $ f(z) $ is defined by
\begin{equation}\label{model_case_definition_f}
	f(z) =
	\begin{cases}
		i \log \varphi(z) - i \log \varphi_+(0) & \textnormal{if } \operatorname{Im}(z) > 0, \\
		- i \log \varphi(z) - i \log \varphi_+(0) & \textnormal{if } \operatorname{Im}(z) < 0,
	\end{cases}
\end{equation}
where $ \varphi_+(0) = \lim_{z \to 0, \operatorname{Im}(z) > 0} \varphi(z) = i $. Now we divide the complex plane into eight congruent octants defined by
\[
	O_i = \bigg\{ z:  \frac{(k-1)\pi}{4} \leq \arg(z) \leq \frac{k \pi}{4} \bigg\}, \quad i = 1, 2, \dots, 8.
\]
Define a $ 2 \times 2 $ matrix valued function $ \psi(z) $ in the first and fourth octants $ O_1 $ and $ O_4 $ by
\begin{equation}\label{model_case_definition_psi}
	\psi(z) =
	\begin{cases}
		\frac{1}{2} \sqrt{\pi} z^{1/2} \begin{pmatrix} e^{-i\frac{2\alpha + 1}{4}\pi} H_{\frac{\alpha + 1}{2}}^{(2)}(z) & -i e^{i\frac{2\alpha + 1}{4}\pi} H_{\frac{\alpha + 1}{2}}^{(1)}(z) \\ e^{-i\frac{2\alpha + 1}{4}\pi} H_{\frac{\alpha - 1}{2}}^{(2)}(z) & -i e^{i\frac{2\alpha + 1}{4}\pi} H_{\frac{\alpha - 1}{2}}^{(1)}(z) \end{pmatrix} & z \in O_1, \\
	\frac{1}{2}\sqrt{\pi}(-z)^{1/2} \begin{pmatrix} i e^{i \frac{2\alpha + 1}{4}\pi} H_{\frac{\alpha + 1}{2}}^{(1)}(-z) & - e^{- i \frac{2\alpha + 1}{4}\pi} H_{\frac{\alpha + 1}{2}}^{(2)}(-z) \\ - i e^{i \frac{2\alpha + 1}{4}\pi} H_{\frac{\alpha - 1}{2}}^{(1)}(-z) & e^{-i \frac{2\alpha + 1}{4}\pi} H_{\frac{\alpha - 1}{2}}^{(2)}(-z) \end{pmatrix} & z \in O_4,

	\end{cases}
\end{equation}
where $ H_{\gamma}^{(1)} $ and $ H_{\gamma}^{(2)} $ denotes the Hankel functions of the first and second kind of order $ \gamma $. For more on the Hankel functions, see \cite[9.1.3, 9.1.4]{Abramowitz-Stegun}. The $ 2 \times 2 $ matrix $ \sigma_3 $ denotes the Pauli matrix
\[
	\sigma_3 = \begin{pmatrix} 1 & 0 \\ 0 & -1 \end{pmatrix}
\]
and for a function $ h(z) $ the symbol $ h(z)^{\sigma_3} $ denotes 
\[
	h(z)^{\sigma_3} = \begin{pmatrix} h(z) & 0 \\ 0 & h(z)^{-1} \end{pmatrix}.
\]
The definition of $ \psi(z) $ can be extended to the whole complex plane, but to avoid complications, we shall define it only on $ O_1 $ and $ O_4 $, because this will be sufficient for our purposes. For the complete definition, see \cite[Section 4.2]{Vanlessen}.
The function $ N(z) $, which is a solution of a different Riemann-Hilbert problem, is defined as
\begin{equation}\label{model_case_definition_N}
	N(z) = D_{\infty}^{\sigma_3} \begin{pmatrix} \frac{a(z) + a(z)^{-1}}{2} & \frac{a(z) - a(z)^{-1}}{2i} \\ \frac{a(z) - a(z)^{-1}}{-2i} & \frac{a(z) + a(z)^{-1}}{2} \end{pmatrix} D(z)^{\sigma_3},
\end{equation}
where $ a(z) = \frac{(z-1)^{1/4}}{(z + 1)^{1/4}} $ and $ D(z) $ is the Szeg\H{o} function (associated to the measure $ \mu_\alpha^b $), which is given by
\[
	D(z) = \frac{z^{\alpha/2}}{\varphi(z)^{\alpha/2}}
\]
and $ D_\infty = \lim_{z \to \infty} D(z) $. Finally we define the auxiliary function $ E_n(z) $ on the first and fourth octant $ O_1 $ and $ O_4 $ as
\begin{equation}\label{model_case_definition_E_n}
	E_n(z) =
	\begin{cases}
		N(z) W(z)^{\sigma_3} e^{\frac{\alpha}{4} \pi i \sigma_3} i^{n \sigma_3} e^{-\frac{\pi}{4} i \sigma_3} \frac{1}{\sqrt{2}} \begin{pmatrix} 1 & i \\ i & 1 \end{pmatrix} & z \in O_1, \\
		N(z) W(z)^{\sigma_3} e^{-\frac{\alpha}{4}\pi i \sigma_3} i^{n \sigma_3} e^{-\frac{\pi}{4} i \sigma_3} \frac{1}{\sqrt{2}} \begin{pmatrix} 1 & i \\ i & 1 \end{pmatrix} & z \in O_4.
	\end{cases}
\end{equation}
$ E_n(z) $ can also be defined on the whole complex plane, but we shall only work on the first and fourth octant. For the complete definition see \cite[Section 4.3]{Vanlessen}. \\

In order to prove (\ref{model_case_bulk_normalized}), we shall need two lemmas which are the analogues of \cite[Lemmas 3.3 and 3.5]{Kuijlaars-Vanlessen}. 

\begin{lemma}\label{model_case_lemma_1}
	For every $ x \in (0, \delta) $, where $ \delta $ is small enough, the first column of $ Y(z) $ given by (\ref{model_case_Y_OP}) takes the form
\begin{equation}\label{model_case_lemma_1_Y_positive}
\begin{aligned}
	\begin{pmatrix} Y_{11}(x) \\ Y_{21}(x) \end{pmatrix} = \sqrt{\frac{\pi n}{w(x)}} & 2^{-n \sigma_3}M_+(x) (\pi/2 - \arccos x)^{1/2} \\
	& \times \begin{pmatrix} e^{-i \frac{\pi}{4}} J_{\frac{\alpha + 1}{2}}\big(n(\pi/2 - \arccos x)\big) \\ e^{-i \frac{\pi}{4}} J_{\frac{\alpha - 1}{2}}\big(n(\pi/2 - \arccos x)\big) \end{pmatrix},
\end{aligned}
\end{equation}
and for every $ x \in (-\delta, 0) $, it takes the form
\begin{equation}\label{model_case_lemma_1_Y_negative}
\begin{aligned}
	\begin{pmatrix} Y_{11}(x) \\ Y_{21}(x) \end{pmatrix} = \sqrt{\frac{\pi n}{w(x)}} & 2^{-n \sigma_3}M_+(x) (\arccos x - \pi/2)^{1/2} \\
	& \times \begin{pmatrix} -e^{-i \frac{\pi}{4}} J_{\frac{\alpha + 1}{2}}\big(n(\arccos x - \pi/2)\big) \\ e^{-i \frac{\pi}{4}} J_{\frac{\alpha - 1}{2}}\big(n(\arccos x - \pi/2)\big) \end{pmatrix},
\end{aligned}
\end{equation}
where $ M(z) $ denotes
\begin{equation}\label{model_case_lemma_1_M}
	M(z) = R(z) E_n(z)
\end{equation}
and $ M_+(x) = \lim_{z \to x, \operatorname{Im}(z) > 0} M(z) $. Moreover, $ \det M(z) \equiv 1 $, $ M(z) $ is analytic in $ O_1 \cup O_4 $, and also $ M(z) $ and $ \frac{d}{dz} M(z) $ are uniformly bounded in $ (O_1 \cup O_4) \cap \{ z: |z| \leq \delta \} $.
\end{lemma}
\begin{proof} Unraveling the series of transformations $ Y \mapsto T \mapsto S \mapsto R $ described in \cite{Vanlessen} it can be obtained that in the first and fourth octant $ O_1, O_4 $ and near the origin, $ Y(z) $ takes the form
\begin{equation}\label{model_case_Y}
	Y(z) = 2^{-n \sigma_3} R(z) E_n(z) \psi(n f(z)) W(z)^{-\sigma_3} \begin{pmatrix} 1 & 0 \\ \frac{1}{w(z)} & 1 \end{pmatrix},
\end{equation}
where $ R(z) $ is analytic and $ R(z) = I  + O(1/n) $ uniformly in a small neighbourhood of the origin. From now on, we have to distinguish between the cases $ x \in (0, \delta) $ and $ x \in (-\delta, 0) $. \\

\emph{First case: $ x \in (0, \delta) $.} To obtain (\ref{model_case_lemma_1_Y_positive}), we work in $ O_1 $ and let $ z \to x $ in there. Since (\ref{W_w_relation}) gives that $ W(z) = w^{1/2}(z) e^{- i \pi \alpha/2} $ for $ z \in O_1 $,
\[
	W(z)^{-\sigma_3} \begin{pmatrix} 1 & 0 \\ \frac{1}{w(z)} & 1 \end{pmatrix} = \begin{pmatrix} w(z)^{-1/2} e^{i\pi \alpha/2} & 0 \\ w(z)^{-1/2} e^{-i\pi \alpha/2} & w(z)^{1/2} e^{-i\pi \alpha/2} \end{pmatrix}.
\]
Combining this with (\ref{model_case_Y}) we obtain that for $ z \in O_1 $,
\[
	\begin{pmatrix} Y_{11}(z) \\ Y_{21}(z) \end{pmatrix} = w(z)^{-1/2} 2^{-n\sigma_3} R(z) E_n(z) \psi(n f(z)) \begin{pmatrix} e^{i \pi \alpha/2} \\ e^{-i \pi \alpha/2} \end{pmatrix}
\]
holds. Now we aim to express $ \psi(n f(z)) \begin{pmatrix} e^{i\pi \alpha/2} \\ e^{-i\pi \alpha/2} \end{pmatrix} $ in terms of Bessel functions. Since $ H_{\nu}^{(1)}(z) + H_{\nu}^{(2)}(z) = 2 J_\nu(z) $, see \cite[9.1.2 and 9.1.3]{Abramowitz-Stegun}, we have
\begin{align*}
	\psi(z) \begin{pmatrix} e^{i\pi \alpha/2} \\ e^{-i\pi \alpha/2} \end{pmatrix} = \sqrt{\pi} z^{1/2} \begin{pmatrix} e^{-i\pi/4} J_{\frac{\alpha + 1}{2}}(z) \\ e^{-i\pi/4} J_{\frac{\alpha - 1}{2}}(z) \end{pmatrix},
\end{align*}
which gives
\[
	\begin{pmatrix} Y_{11}(z) \\ Y_{21}(z) \end{pmatrix} = w(z)^{-1/2} \sqrt{\pi} 2^{-n\sigma_3} R(z) E_n(z) (nf(z))^{1/2} \begin{pmatrix} e^{-i\pi/4} J_{\frac{\alpha + 1}{2}}(nf(z)) \\ e^{-i\pi/4} J_{\frac{\alpha - 1}{2}}(nf(z)) \end{pmatrix}.
\]
For the $ f(z) $ defined by (\ref{model_case_definition_f}) we have $ f_+(x) = \pi/2 - \arccos x $, therefore the above identity gives (\ref{model_case_lemma_1_Y_positive}).
Since $ \det R(z) \equiv 1 $ (use that $ \det R(z) $ is analytic and $ R(z) = I + O(1/z) $ around infinity, see \cite[Section 3.3]{Vanlessen}), it follows easily from (\ref{model_case_definition_N}) that $ \det M(z) \equiv 1 $. Moreover, $ M(z) $ is analytic in the octant $ O_1 $ near the origin. The boundedness of $ R(z) $ is implied by \cite[(3.30)]{Vanlessen}, and since it is actually analytic in some small neighbourhood of $ 0 $, the Cauchy integral formula implies that $ \frac{d}{dz} R(z) $ is also bounded in some small disk with center at the origin. The same can be said about $ E_n(z) $, see \cite[Proposition 4.5]{Vanlessen}, which implies that $ M(z) = R(z) E_n(z) $ and $ \frac{d}{dz} M(z) $ is uniformly bounded in a small disk around $ 0 $ as $ n \to \infty $. \\

\emph{Second case: $ x \in (-\delta, 0) $.} Calculating similarly as in the first case, we obtain that for $ z \in O_4 $, we have
\[
	\begin{pmatrix} Y_{11}(z) \\ Y_{21}(z) \end{pmatrix} = w(z)^{-1/2} 2^{-n\sigma_3} R(z) E_n(z) \psi(n f(z)) \begin{pmatrix} e^{-i \pi \alpha/2} \\ e^{i \pi \alpha/2} \end{pmatrix}.
\]
We also get
\[
	\psi(z) \begin{pmatrix} e^{-i \pi \alpha/2} \\ e^{i \pi \alpha/2} \end{pmatrix} = \sqrt{\pi} (-z)^{1/2} \begin{pmatrix} -e^{-i \frac{\pi}{4}} J_{\frac{\alpha+1}{2}}(-z) \\ e^{-i \frac{\pi}{4}}J_{\frac{\alpha - 1}{2}}(-z) \end{pmatrix},
\]
which gives
\[
	\begin{pmatrix} Y_{11}(z) \\ Y_{12}(z) \end{pmatrix} = w(z)^{-1/2} \sqrt{\pi} 2^{-n\sigma_3} R(z)E_n(z) (-nf(z))^{1/2} \begin{pmatrix} -e^{-i \frac{\pi}{4}} J_{\frac{\alpha+1}{2}}(-nf(z)) \\ e^{-i \frac{\pi}{4}}J_{\frac{\alpha - 1}{2}}(-nf(z)). \end{pmatrix}
\]
Similarly as in the first case, letting $ z \to x $ through $ O_4 $ yields (\ref{model_case_lemma_1_Y_negative}).
\end{proof}

The next lemma, which is the analogue of \cite[Lemma 3.5]{Kuijlaars-Vanlessen}, studies the asymptotic behavior of $ J_\nu\big(n\big( \frac{\pi}{2} - \arccos \frac{a}{n} \big)\big) $.

\begin{lemma}\label{model_case_lemma_2}
Let $ a \in \mathbb{R} \setminus \{ 0 \} $. Define $ a_n = a/n $ and $ \tilde{a}_n  = n \big(\frac{\pi}{2} - \arccos |a_n| \big) $. Then
\begin{equation}\label{model_case_lemma_2_1}
	\tilde{a}_n = |a| + O\bigg( \frac{|a|^3}{n^2} \bigg)
\end{equation}
and
\begin{equation}\label{model_case_lemma_2_2}
	J_\alpha(\tilde{a}_n) = J_\alpha(|a|) + O\bigg( \frac{|a|^{2 + \alpha}}{n^2} \bigg)
\end{equation}
holds.
\end{lemma}
\begin{proof}
Without the loss of generality we can assume that $ a > 0 $. Since $ \pi/2 - \arccos x = x + O(x^3) $ (just use the Taylor expansion of $ \pi /2 - \arccos x = \arcsin x $), the asymptotic formula (\ref{model_case_lemma_2_1}) easily follows. As for the behavior of the Bessel functions, \cite[9.1.10]{Abramowitz-Stegun} says that $ J_\alpha(z) = z^\alpha G(z) $, where $ G(z) $ is an entire function. It follows that
\begin{align*}
	J_\alpha(\tilde{a}_n) & = \bigg( a + O\bigg( \frac{a^3}{n^2} \bigg) \bigg)^\alpha \bigg( G(a) + O \bigg( \frac{a^3}{n^2} \bigg) \bigg) \\
	& = a^\alpha \bigg( 1 + O\bigg( \frac{a^2}{n^2} \bigg) \bigg) \bigg( G(a) + O \bigg( \frac{a^3}{n^2} \bigg) \bigg) \\
	& = J_\alpha(a) + O\bigg( \frac{a^{2 + \alpha}}{n^2} \bigg),
\end{align*}
which is what we needed to show.
\end{proof}

To show (\ref{model_case_bulk_normalized}), we shall distinguish between the four cases (i) $ a,b \geq 0 $, (ii) $ a \geq 0, b < 0 $, (iii) $ a < 0, b \geq 0 $, (iv) $ a,b < 0 $. Because of the symmetry of $ \mu_{\alpha}^{b} $, we have $ K_n(x,y) = K_n(-x,-y) $, therefore it is enough to deal with the first two cases. \\

\emph{First case: $ a, b \geq 0 $.} Let $ a,b \in (0, \infty) $ and define $ a_n = a/n $, $ b_n = b/n $, moreover let $ \tilde{a}_n  = n \big(\frac{\pi}{2} - \arccos a_n \big) $, $ \tilde{b}_n  = n \big(\frac{\pi}{2} - \arccos b_n \big) $. First we shall express  $ \widetilde{K}_{n}(x,y) $ in terms of $ Y_{11}(x) $ and $ Y_{21}(x) $. (\ref{model_case_Y_OP}) gives that $ Y_{11}(x) = \frac{1}{\gamma_n}p_n(x) $ and $ Y_{21}(x) = -2\pi i \gamma_{n - 1} p_{n-1}(x) $. Using the Christoffel-Darboux formula (\ref{Christoffel-Darboux}) with Lemma \ref{model_case_lemma_1} yields
\begin{align*}
	\frac{1}{n} \widetilde{K}_{n}(a_n &, b_n) = \frac{1}{2\pi i (b - a)} \sqrt{w(a_n) w(b_n)} \Big( Y_{11}(a_n) Y_{21}(b_n) - Y_{11}(b_n) Y_{21}(a_n) \Big) \\
	 & = \frac{1}{2\pi i (b - a)} \sqrt{w(a_n) w(b_n)} \det \begin{pmatrix} Y_{11}(a_n) & Y_{11}(b_n) \\ Y_{21}(a_n) & Y_{21}(b_n) \end{pmatrix} \\
	& = \frac{1}{2\pi i (b - a)} \det \begin{pmatrix} \sqrt{w(a_n)} Y_{11}(a_n) & \sqrt{w(b_n)} Y_{11}(b_n) \\ \sqrt{w(a_n)} Y_{21}(a_n) & \sqrt{w(b_n)} Y_{21}(b_n) \end{pmatrix} \\
	 & = \frac{n}{2(a - b)} \det \Bigg[ M_+(a_n) (\tilde{a}_n / n)^{1/2} \begin{pmatrix} J_{\frac{\alpha + 1}{2}}(\tilde{a}_n) & 0 \\ J_{\frac{\alpha - 1}{2}}(\tilde{a}_n) & 0 \end{pmatrix} \\ 
	& \phantom{aaaaaaaaaaaaa} + M_+(b_n) (\tilde{b}_n / n)^{1/2} \begin{pmatrix} 0 & J_{\frac{\alpha + 1}{2}}(\tilde{b}_n) \\ 0 & J_{\frac{\alpha - 1}{2}}(\tilde{b}_n) \end{pmatrix} \Bigg] \\
\end{align*}
\begin{align*}
	\phantom{\frac{1}{n} \widetilde{K}_{n}(a_n}
	& = \frac{n}{2(a - b)} \det \Bigg[ M_+(b_n) \Bigg\{ \begin{pmatrix} (\tilde{a}_n / n)^{1/2} J_{\frac{\alpha + 1}{2}}(\tilde{a}_n) & (\tilde{b}_n / n)^{1/2} J_{\frac{\alpha + 1}{2}}(\tilde{b}_n) \\ (\tilde{a}_n / n)^{1/2} J_{\frac{\alpha - 1}{2}}(\tilde{a}_n) & (\tilde{b}_n / n)^{1/2} J_{\frac{\alpha - 1}{2}}(\tilde{b}_n) \end{pmatrix}  \\
	& \phantom{\begin{pmatrix} 0 \\ 0 \end{pmatrix}aaaaaaaaaa} + M_+(b_n)^{-1} (M_+(a_n) - M_+(b_n)) \\
	& \phantom{aaaaaaaaaaaaaaaaaa} \times \begin{pmatrix} (\tilde{a}_n / n)^{1/2} J_{\frac{\alpha + 1}{2}}(\tilde{a}_n) & 0 \\ (\tilde{a}_n / n)^{1/2} J_{\frac{\alpha - 1}{2}}(\tilde{a}_n) & 0 \end{pmatrix} \Bigg\} \Bigg].
\end{align*}
Since $ M(z) $ is uniformly bounded and its determinant is $ 1 $ (see Lemma \ref{model_case_lemma_1}), $ M(z)^{-1} $ is also uniformly bounded, moreover the uniform boundedness of $ \frac{d}{dz} M(z) $ imply $ M_+(a_n) - M_+(b_n) = O( \frac{a - b}{n} ) $. We also have $ J_\alpha(\tilde{a}_n) = O(a^\alpha) $ (use (\ref{model_case_lemma_2_2}) and the fact that $ J_\alpha(z) = z^\alpha G(z) $, where $ G(z) $ is an entire function), which gives
\[
	M_+(b_n)^{-1} (M_+(a_n) - M_+(b_n))\begin{pmatrix} (\tilde{a}_n / n)^{1/2} J_{\frac{\alpha + 1}{2}}(\tilde{a}_n) & 0 \\ (\tilde{a}_n / n)^{1/2} J_{\frac{\alpha - 1}{2}}(\tilde{a}_n) & 0 \end{pmatrix} = \begin{pmatrix} O\big( \frac{a-b}{n^{3/2}} a^{\frac{\alpha+2}{2}} \big) & 0 \\ O\big( \frac{a-b}{n^{3/2}} a^{\frac{\alpha}{2}} \big) & 0 \end{pmatrix}.
\]
From these and $ \det M(z) \equiv 1 $ it follows that
\begin{equation}\label{model_case_K_tilde_det}
\begin{aligned}
	\frac{1}{n} & \widetilde{K}_{n}(a_n, b_n) = \frac{1}{2(a -b )} \det \begin{pmatrix} \tilde{a}_n^{1/2} J_{\frac{\alpha + 1}{2}}(\tilde{a}_n) + O\big( \frac{a-b}{n} a^{\frac{\alpha+2}{2}} \big) & \tilde{b}_n^{1/2} J_{\frac{\alpha + 1}{2}}(\tilde{b}_n) \\ \tilde{a}_n^{1/2} J_{\frac{\alpha - 1}{2}}(\tilde{a}_n) + O\big( \frac{a-b}{n} a^{\frac{\alpha}{2}} \big) & \tilde{b}_n^{1/2} J_{\frac{\alpha - 1}{2}}(\tilde{b}_n) \end{pmatrix} \\
	& = \frac{\tilde{a}_n^{1/2}\tilde{b}_n^{1/2}}{2(a -b)} \det \begin{pmatrix} J_{\frac{\alpha + 1}{2}}(\tilde{a}_n) & J_{\frac{\alpha + 1}{2}}(\tilde{b}_n) \\ J_{\frac{\alpha - 1}{2}}(\tilde{a}_n) & J_{\frac{\alpha - 1}{2}}(\tilde{b}_n) \end{pmatrix} + O \bigg( \frac{ a^{\alpha/2} b^{\alpha/2}}{n} \bigg) \\
	& = \frac{\tilde{a}_n^{1/2}\tilde{b}_n^{1/2} a^{\frac{\alpha - 1}{2}} b^{\frac{\alpha - 1}{2}} }{2(a -b)} \det \begin{pmatrix} a^{-\frac{\alpha - 1}{2}} J_{\frac{\alpha + 1}{2}}(\tilde{a}_n) - b^{-\frac{\alpha - 1}{2}} J_{\frac{\alpha + 1}{2}}(\tilde{b}_n) & b^{-\frac{\alpha - 1}{2}} J_{\frac{\alpha + 1}{2}}(\tilde{b}_n) \\ a^{-\frac{\alpha - 1}{2}} J_{\frac{\alpha - 1}{2}}(\tilde{a}_n) - b^{-\frac{\alpha - 1}{2}} J_{\frac{\alpha - 1}{2}}(\tilde{b}_n) & b^{-\frac{\alpha - 1}{2}} J_{\frac{\alpha - 1}{2}}(\tilde{b}_n)  \end{pmatrix} \\
	& \phantom{aaaaaaaaaaaaaaaa} + O \bigg( \frac{ a^{\alpha/2} b^{\alpha/2}}{n} \bigg),
\end{aligned}
\end{equation}
where the error term is uniform for $ a, b $ on bounded subsets of $ (0, \infty) $, even intervals of the form $ (0,c] $. Lemma \ref{model_case_lemma_2} gives
\[
	a^{-\frac{\alpha - 1}{2}} \big( J_{\frac{\alpha + 1}{2}}(\tilde{a}_n) - J_{\frac{\alpha + 1}{2}}(a) \big) = O \bigg( \frac{a^3}{n^2} \bigg)
\]
and
\[
	a^{-\frac{\alpha - 1}{2}} \big( J_{\frac{\alpha - 1}{2}}(\tilde{a}_n) - J_{\frac{\alpha - 1}{2}}(a) \big) = O \bigg( \frac{a^2}{n^2} \bigg),
\]
which imply
\begin{align*}
	a^{-\frac{\alpha - 1}{2}} J_{\frac{\alpha + 1}{2}}(\tilde{a}_n) & - b^{-\frac{\alpha - 1}{2}} J_{\frac{\alpha + 1}{2}}(\tilde{b}_n) \\
	&= a^{-\frac{\alpha - 1}{2}} J_{\frac{\alpha + 1}{2}}(a) - b^{-\frac{\alpha - 1}{2}} J_{\frac{\alpha + 1}{2}}(b) + O\bigg( \frac{a^3 - b^3}{n^2} \bigg)
\end{align*}
and
\begin{align*}
	a^{-\frac{\alpha - 1}{2}} J_{\frac{\alpha - 1}{2}}(\tilde{a}_n) & - b^{-\frac{\alpha - 1}{2}} J_{\frac{\alpha - 1}{2}}(\tilde{b}_n) \\
	& = a^{-\frac{\alpha - 1}{2}} J_{\frac{\alpha - 1}{2}}(a) - b^{-\frac{\alpha - 1}{2}} J_{\frac{\alpha - 1}{2}}(b) + O\bigg( \frac{a^2 - b^2}{n^2} \bigg).
\end{align*}
Continuing (\ref{model_case_K_tilde_det}) with these, we have
\begin{align*}
	\frac{1}{n} \widetilde{K}_{n}(a_n, b_n) & = \frac{\tilde{a}_n^{1/2} \tilde{b}_n^{1/2}}{2(a - b)} \det \begin{pmatrix} J_{\frac{\alpha + 1}{2}}(a) & J_{\frac{\alpha + 1}{2}}(b) \\ J_{\frac{\alpha - 1}{2}}(a) & J_{\frac{\alpha - 1}{2}}(b) \end{pmatrix} \\
	& + \frac{\tilde{a}_n^{1/2} \tilde{b}_n^{1/2} a^{\frac{\alpha - 1}{2}} b^{\frac{\alpha - 1}{2}}}{2(a - b)} \det \begin{pmatrix} a^{-\frac{\alpha - 1}{2}} J_{\frac{\alpha + 1}{2}}(a) - b^{-\frac{\alpha - 1}{2}} J_{\frac{\alpha + 1}{2}}(b) & O \Big( \frac{b^2}{n^2} \Big) \\ a^{-\frac{\alpha - 1}{2}} J_{\frac{\alpha - 1}{2}}(a) - b^{-\frac{\alpha - 1}{2}} J_{\frac{\alpha - 1}{2}}(b) & O \Big( \frac{b^2}{n^2} \Big) \end{pmatrix} \\
	& + O \bigg( \frac{ a^{\alpha/2} b^{\alpha/2}}{n} \bigg)
\end{align*}
In the second term, since $ J_\nu(z) = z^\nu G(z) $ where $ G(z) $ is an entire function, by the mean value theorem we have
\begin{equation}\label{model_case_a_b_positive}
	\frac{a^{-\frac{\alpha - 1}{2}} J_{\frac{\alpha \pm 1}{2}}(a) - b^{-\frac{\alpha - 1}{2}} J_{\frac{\alpha \pm 1}{2}}(b)}{a-b} = O(1),
\end{equation}
hence the second term is $ O\Big(\frac{a^{\alpha/2} b^{\alpha/2}}{n} \Big) $. In the first term $ \tilde{a}_n = a + O\big( \frac{a^3}{n^2} \big) $ and $ \tilde{b}_n = b + O\big( \frac{b^3}{n^2} \big) $ can be replaced with $ a $ and $ b $ respectively, because the resulting error terms can be absorbed into the previous error term. Overall, we have
\begin{align*}
	\frac{1}{n} \widetilde{K}_n\bigg( \frac{a}{n}, \frac{b}{n} \bigg) = \mathbb{L}_\alpha(a, b) + O\bigg(\frac{a^{\alpha/2} b^{\alpha/2}}{n} \bigg),
\end{align*}
which is uniform for $ a, b \in (0, \infty) $ in bounded sets, even when $ a - b $ is small. \\

\emph{Second case: $ a \geq 0, b < 0 $.} Let $ a \geq 0, b < 0 $ and define $ a_n = a/n $, $ b_n = b/n $, moreover let $ \tilde{a}_n  = n \big(\frac{\pi}{2} - \arccos a_n \big) $, $ \tilde{b}_n = n \big( \arccos b_n - \frac{\pi}{2} \big) = n \big( \frac{\pi}{n} - \arccos |b_n| \big) $. With similar calculations as before but with (\ref{model_case_lemma_1_Y_negative}) instead of (\ref{model_case_lemma_1_Y_positive}) we obtain (\ref{model_case_a_b_positive}). (Note that the definition of $ \mathbb{L}_\alpha(a,b) $ differs when $ a \geq 0 $ and $ b < 0 $, see (\ref{kernel_function}).) \\

By normalizing (\ref{model_case_bulk_normalized}) with $ a^{\alpha/2} b^{\alpha/2} $, we obtain (\ref{model_case_bulk}) uniformly for $ a, b $ in bounded subsets of $ \mathbb{R} \setminus \{0\} $. Using \cite[Theorem 1.1]{Danka-Totik} we see that this actually holds uniformly for $ a, b $ in compact subsets of $ \mathbb{R} $, and this is what we had to show.
\end{proof}

By letting $ b \to a $ in (\ref{model_case_bulk}) we obtain the formula
\[
	\frac{1}{n^{\alpha + 1}} K_n\bigg( \frac{a}{n}, \frac{a}{n} \bigg) = \mathbb{L}_\alpha^*(a) + O(1/n),
\]
which can be written in terms of Christoffel functions as
\begin{equation}\label{model_case_bulk_Christoffel}
	\lim_{n \to \infty} n^{\alpha + 1} \lambda_n(\mu_\alpha^b, x_0 + a/n) = \Big( \mathbb{L}_\alpha^*(a) \Big)^{-1}
\end{equation}
We shall generalize this in Section \ref{section_Christoffel_functions} for measures supported on arbitrary compact sets, and in turn this will serve a basic building block towards universality limits.

\subsection{Model case for the edge}

The measure $ \mu_\alpha^e $, which is defined by (\ref{mu_alpha_edge}), has been extensively studied by various authors and asymptotic formulas are already established. It was shown by Kuijlaars and Vanlessen in \cite{Kuijlaars-Vanlessen} that
\[
	\frac{1}{2n^2} \widetilde{K}_n\bigg( 1 - \frac{a}{2n^2}, 1 - \frac{b}{2n^2} \bigg) = \mathbb{J}_\alpha(a,b) + O\bigg( \frac{a^{\alpha/2} b^{\alpha/2}}{n} \bigg)
\]
uniformly for $ a, b $ in bounded subsets of $ (0, \infty) $, where $ \mathbb{J}_\alpha(a,b) $ is the so-called Bessel kernel defined by (\ref{bessel_kernel}). For $ a = b $, this limit can also be written in terms of Christoffel functions as
\begin{equation}\label{model_case_edge_Christoffel}
	\lim_{n \to \infty} n^{2\alpha + 2} \lambda_n\bigg(\mu_\alpha^e, 1 - \frac{a}{2n^2} \bigg) = \Big( 2^{\alpha + 1} \mathbb{J}_\alpha^*(a) \Big)^{-1},
\end{equation}
which holds uniformly for $ a $ in compact subsets of the real line, where $ \mathbb{J}_\alpha^*(a) $ is defined by $ \mathbb{J}_\alpha(a,a)/a^\alpha $.

\section{Christoffel functions}\label{section_Christoffel_functions}

Now we aim to generalize the results obtained in Section \ref{section_model_cases} for measures supported on arbitrary compact sets. We do this with the method of polynomial inverse images, which was developed by Totik in \cite{Totik_2} to prove polynomial inequalities on several intervals. Let $ T_N $ be a polynomial of degree $ N $ such that all of its zeros are real and simple, moreover suppose that if $ T_{N}^{\prime}(x_0) = 0 $ for some $ x_0 $ (i.e. $ x_0 $ is a local extrema for $ T_N$), then $ |T_N(x_0)| \geq 1 $. Such polynomials are called \emph{admissible}. The inverse image of the interval $ [-1,1] $ defined as $ E_N = T_N^{-1}([-1,1]) $ taken with respect to an admissible polynomial is a finite union of intervals, i.e. $ T_N^{-1}([-1,1]) = \cup_{k=0}^{N-1} [a_k,b_k] $, where $ T_N $ restricted to $ [a_k, b_k] $ is a bijection between $ [a_k, b_k] $ and $ [-1,1] $. These polynomial inverse images have good approximation properties which we shall use frequently, see Lemmas \ref{bulk_approximation} and \ref{edge_approximation} below.

For every integrable function $ f $ and for each $ k \in \{ 0, 1, \dots, N-1 \} $ the formula
\begin{equation}\label{general_lemniscate_integral}
\begin{aligned}
	\int_{-1}^{1} f(x) dx & = \int_{a_k}^{b_k} f(T_N(x)) |T_N^{\prime}(x)| dx \\
	& = \frac{1}{N} \int_{E_N} f(T_N(x)) |T_N^{\prime}(x)| dx
\end{aligned}
\end{equation}
also holds. For example, if $ P_n $ is a polynomial of degree $ N $, we have
\begin{equation}\label{lemniscate_integral}
\begin{aligned}
	\int_{-1}^{1} |P_n(x)|^2 |x|^\alpha dx & = \int_{a_k}^{b_k} |P_n(T_N(x))|^2 |T_N(x)|^\alpha |T_N^{\prime}(x)| dx \\
	& = \frac{1}{N} \int_{E_N} |P_n(T_N(x))|^2 |T_N(x)|^\alpha |T_N^{\prime}(x)| dx,
\end{aligned}
\end{equation}
which will be especially useful to us. In addition, the equilibrium density for $ E_N $ is given by the formula
\begin{equation}\label{inverse_image_equilibrium_general}
	\omega_{E_N}(x) = \frac{|T_{N}^{\prime}(x)|}{N \pi \sqrt{1-T_N(x)^2}},
\end{equation}
see \cite{Totik_2} for details. (Or use (\ref{general_lemniscate_integral}) and the fact that the equilibrium measure on $ E_N $ is the pullback of the equilibrium measure on $ [-1,1] $ with respect to the mapping $ T_N $.)

\subsection{Preliminary estimates}

When establishing asymptotics for $ \lambda_n(\mu, x) $, it is enough to study special subsequences from which the asymptotic behavior of the complete sequence is implied. This is summarized in the following lemma, which will be used frequently.

\begin{lemma}\label{subsequence_lemma}
Let $ \{ n_k \}_{k=1}^{\infty} $ be a subsequence of $ \mathbb{N} $ such that $ n_{k+1}/n_k \to 1 $ as $ k \to \infty $. Then for every $ \kappa > 0 $,
\[
	\liminf_{k \to \infty} n_k^\kappa \lambda_{n_k}(\mu, x) = \liminf_{n \to \infty} n^\kappa \lambda_{n}(\mu, x)
\]
and
\[
	\limsup_{k \to \infty} n_k^\kappa \lambda_{n_k}(\mu, x) = \limsup_{n \to \infty} n^\kappa \lambda_{n}(\mu, x)
\]
holds.
\end{lemma}
\begin{proof}
If $ k $ is selected such that $ n_k \leq n \leq n_{k+1} $, the monotonicity of $ \lambda_n(\mu, x) $ in $ n $ implies
\[
	\bigg( \frac{n}{n_k} \bigg)^\kappa n_k^\kappa \lambda_{n_k}(\mu, x) \leq n^\kappa \lambda_n(\mu, x) \leq \bigg( \frac{n}{n_{k+1}} \bigg)^\kappa n_{k+1}^\kappa \lambda_{n_{k+1}}(\mu, x).
\]
Since $ n/n_k \to 1 $ as $ k \to \infty $, this implies what we need to prove.
\end{proof}

In order to study the asymptotic behavior of $ \lambda_n(\mu, x_0 + a/n) $, we shall need a tool to control small perturbations. The next lemma will be useful for studying Christoffel functions in the bulk of the support.

\begin{lemma}\label{perturbation_lemma}
Let $ \mu $ be a finite Borel measure and suppose that $ \mu $ is supported on a compact set $ K = \operatorname{supp}(\mu) $ on the real line. Let $ x_0 \in \operatorname{int}(K) $ be a point from the interior of its support and suppose that for some $ \varepsilon > 0 $ the measure $ \mu $ is absolutely continuous on $ (x_0 - \varepsilon, x_0 + \varepsilon) $ with
\[
	d\mu(x) = w(x) |x - x_0|^\alpha dx, \quad x \in (x_0 - \varepsilon, x_0 + \varepsilon)
\]
there for some $ \alpha > -1 $, where $ w $ is strictly positive and continuous at $ x_0 $. Then for a given sequence $ \varepsilon_n = o(n^{-1}) $,
\[
	\lim_{n \to \infty} \frac{\lambda_n(\mu, x_0 + a/n)}{\lambda_n(\mu, x_0 + a/n + \varepsilon_n)} = 1 
\]
holds uniformly for $ a \in \mathbb{R} $ in compact subsets of the real line.
\end{lemma}
\begin{proof}
During the proof, constants are denoted with $ C $ and their value often varies from line to line. We can assume without the loss of generality that $ x_0 = 0 $. The classical bound of Nevai \cite[p. 120 Theorem 28]{Nevai_2} says that there is a constant $ C $ independent of $ x $ such that
\begin{equation}\label{Nevai_bound_bulk}
	\frac{1}{C n} \bigg( |x| + \frac{1}{n} \bigg)^{\alpha} \leq \lambda_n(\mu_{\alpha}^{b}, x) \leq \frac{C}{n} \bigg( |x| + \frac{1}{n} \bigg)^{\alpha}, \quad x \in (-1/2, 1/2)
\end{equation}
holds, where $ \mu_\alpha^b $ is defined by (\ref{mu_alpha_bulk}). We wish to establish the same bounds for $ \lambda_n(\mu, x) $. Let $ \delta > 0 $ be so small such that $ \delta < \varepsilon $ and
\[
	\frac{w(0)}{2} \leq w(x) \leq 2 w(0), \quad x \in (-\delta, \delta)
\]
holds. Suppose that $ P_n(x) $ is extremal for $ \lambda_n(\mu, x_1) $ for some $ x_1 \in (-\delta, \delta) $, i.e. $ P_n(x) $ is a polynomial of degree less than $ n $ with $ P_n(x_1) = 1 $ and $ \int |P_n|^2 d\mu = \lambda_n(\mu, x_1) $. Then
\[
	\lambda_n(\mu, x_1) \geq \frac{w(0)}{2} \int_{-\delta}^{\delta} |P_n(x)| |x|^\alpha dx \geq \frac{w(0)}{2} \lambda_n\big( |x|^\alpha \chi_{[-\delta, \delta]}(x) dx, x_1 \big),
\]
where $ \chi_H(x) $ denotes the characteristic function of the set $ H $. After scaling the measures appropriately, the bound (\ref{Nevai_bound_bulk}) can be applied for the quantity $ \lambda_n\big( |x|^\alpha \chi_{[-\delta, \delta]}(x) dx, x_1 \big) $, thus there is a constant $ C $ such that
\begin{equation}\label{Nevai_general_lower_bound_bulk}
	\lambda_n(\mu, x_1) \geq \frac{1}{C n} \bigg( |x_1| + \frac{1}{n} \bigg)^{\alpha}, \quad x_1 \in (-\delta/2, \delta/2).
\end{equation}
On the other hand, let $ b > 0 $ be so large such that $ K \subseteq [-b,b] $, let $ P_n(x) $ be extremal for $ \lambda_n\big( |x|^\alpha \chi_{[-b,b]}(x) dx, x_1 \big) $ and define the polynomial
\[
	R_n(x) = P_n(x) \bigg( 1 - \frac{(x-x_1)^2}{2b} \bigg)^{n}.
\]
$ R_n(x) $ is a polynomial of degree at most $ 3n $, moreover $ R_n(x_1) = 1 $ and $ |R_n(x)| \leq |P_n(x)| $ for all $ x \in [-b,b] $. Then
\begin{equation}\label{perturbation_upper_estimate_1}
\begin{aligned}
	\lambda_{3n}(\mu, x_1) & \leq \int |R_n(x)|^2 d\mu(x) \\
	& \leq 2 w(0) \int_{-\delta}^{\delta} |P_n(x)|^2 |x|^\alpha dx + \gamma^n \int_{K \setminus [-\delta, \delta]} |P_n(x)|^2 d\mu(x) \\
	& \leq 2 w(0) \lambda_n(|x|^\alpha \chi_{[-b,b]}(x)dx, x_1) + \mu(K) \gamma^n \| P_n \|_K,
\end{aligned}
\end{equation}
where
\[
	\gamma = \sup_{x \in K \setminus [-\delta, \delta]} \bigg| 1 - \frac{(x - x_1)^2}{2\operatorname{diam}(K)} \bigg|^n < 1.
\]
As the Erd\H{o}s-Tur\'an criterion implies, the measure $ |x|^\alpha \chi_{[-b,b]} dx $ is regular in the sense of Stahl-Totik moreover its support $ [-b,b] $ is regular with respect to the Dirichlet problem. Then (\ref{reg_dirichlet}) gives that for every $ \tau > 0 $,
\[
	\| P_n \|_{[-b,b]} \leq (1 + \tau)^n \lambda_{n}\big( |x|^\alpha \chi_{[-b,b]}(x)dx, x_1 \big)^{1/2}
\]
holds if $ n $ is large enough. Since $ \lambda_{n}\big( |x|^\alpha \chi_{[-b,b]}(x)dx, x_1 \big) = O(1) $ (use the constant polynomial $ 1 $ in the definition (\ref{Christoffel_definition})), then it follows that if $ \tau $ is chosen such that $ q = \gamma (1 + \tau) < 1 $, we have
\[
	\mu(K) \gamma^n \| P_n \|_K = O(q^n).
\]
This together with Lemma \ref{subsequence_lemma},  (\ref{Nevai_bound_bulk}) and (\ref{perturbation_upper_estimate_1}) implies that there is a constant $ C $ such that
\begin{equation}\label{Nevai_general_upper_bound_bulk}
	\lambda_{n}(\mu, x_1) \leq \frac{C}{n} \bigg( |x_1| + \frac{1}{n} \bigg)^\alpha, x_1 \in (-\delta/2, \delta/2).
\end{equation}
Now define the polynomial $ Q_n(x) $ as
\[
	Q_n(x) = \frac{\lambda_n(\mu, a/n)}{\lambda_n(\mu, a/n + x)}.
\]
$ Q_n(x) $ is indeed a polynomial of degree $ 2n - 2 $ as implied by (\ref{Christoffel_polynomial_form}), moreover $ Q_n(0) = 1 $. (\ref{Nevai_general_lower_bound_bulk}) and (\ref{Nevai_general_upper_bound_bulk}) gives that
\begin{equation}\label{perturbation_lemma_P_bound_1}
	|Q_n(x)| \leq C \bigg( \frac{|a/n| + 1/n}{|a/n + x| + 1/n} \bigg)^{\alpha} \leq C, \quad x \in [-\delta/4, \delta/4]
\end{equation}
that is $ Q_n(x) $ is bounded on the small but fixed interval $ [-\delta/4, \delta/4] $, moreover the bound holds uniformly for $ a $ in compact subsets of the real line. The iterated Bernstein inequality for $ [-\delta/4, \delta/4] $, see \cite[p. 260 Exercise 5e]{Borwein-Erdelyi}, gives that
\begin{equation}\label{lambda_polynomial_bound}
	|Q_n^{(k)}(0)| \leq C M^k n^k
\end{equation}
holds for some constants $ C $ and $ M $.
Overall, since $ \varepsilon_n = o(n^{-1}) $, we have
\begin{equation}\label{perturbation_lemma_P_bound_2}
	|Q_n(\varepsilon_n)| \leq \sum_{k=0}^{2n-2} \frac{|Q_n^{(k)}(0)|}{k!} \varepsilon_n^k \leq 1 + C \sum_{k=1}^{2n-2} \frac{M^k n^k \varepsilon_n^k}{k!} \leq 1 + o(1),
\end{equation}
and since (\ref{perturbation_lemma_P_bound_1}) holds uniformly for $ a $ in compact subsets of the real line, the above bound is also uniform. This implies
\[
	\limsup_{n \to \infty} \frac{\lambda_n(\mu, a/n)}{\lambda_n(a/n + \varepsilon_n)} \leq 1,
\]
which is half of what we need. To obtain the matching estimate
\[
	\limsup_{n \to \infty} \frac{\lambda_n(a/n + \varepsilon_n)}{\lambda_n(\mu, a/n)} \leq 1,
\]
define the polynomial
\[
	Q_n(x) = \frac{\lambda_n(a/n + \varepsilon_n)}{\lambda_n(\mu, a/n + \varepsilon_n + x)}
\]
and repeat the argument given in (\ref{perturbation_lemma_P_bound_1}) - (\ref{perturbation_lemma_P_bound_2}) to see that $ |Q_n(-\varepsilon_n)| \leq 1 + o(1) $ holds.
\end{proof}

The analogue of the previous lemma for the edge of the support is the following.

\begin{lemma}\label{perturbation_lemma_edge}
Let $ \mu $ be a finite Borel measure and suppose that $ \mu $ is supported on a compact set $ K = \operatorname{supp}(\mu) $ on the real line. Let $ x_0 \in K $ be a right endpoint of $ K $ (i.e. $ K \cap (x_0, x_0 + \varepsilon_0) = \varnothing $ for some $ \varepsilon_0 > 0 $) and assume that for some $ \varepsilon > 0 $ the measure $ \mu $ is absolutely continuous on $ (x_0 - \varepsilon, x_0] $ with
\[
	d\mu(x) = w(x) |x - x_0|^\alpha dx, \quad x \in (x_0 - \varepsilon, x_0]
\]
there for some $ \alpha > -1 $, where $ w $ is strictly positive and left continuous at $ x_0 $. Then for a given sequence $ \varepsilon_n = o(n^{-2}) $, 
\[
	\lim_{n \to \infty} \frac{\lambda_n(\mu, x_0 - a/n^2)}{\lambda_n(\mu, x_0 - a/n^2 + \varepsilon_n)} = 1
\]
holds for all $ a \in [0, \infty) $.
\end{lemma}
\begin{proof}
The proof follows in a similar tune to Lemma \ref{perturbation_lemma}, with a few differences. Without the loss of generality we can assume that $ x_0 = 1 $. Since $ \varepsilon_n = o(n^{-2}) $, we can also assume that $ |\varepsilon_n| \leq a/n^2 $. Let $ \delta > 0 $ be so small such that
\[
	\frac{w(1)}{2} \leq w(x) \leq 2 w(1), \quad x \in (1 - \delta, 1]
\]
holds and $ K \cap (1,1+\delta) = \varnothing $. The classical bound of Nevai \cite[p. 120 Theorem 28]{Nevai_2} once more says that there is a constant $ C $ independent of $ x $ such that
\[
	\frac{1}{C n} \bigg( \sqrt{1 - x} + \frac{1}{n} \bigg)^{2\alpha + 1} \leq \lambda_n(\mu_\alpha^e, x) \leq \frac{C}{n} \bigg( \sqrt{1 - x} + \frac{1}{n} \bigg)^{2\alpha + 1}, \quad x \in (1/2, 1]
\]
holds.  Similarly like in the proof of Lemma \ref{perturbation_lemma}, we shall show that this holds if we replace $ \mu_\alpha^e $ with $ \mu $. The proof of the lower estimate
\[
	\frac{1}{C n} \bigg( \sqrt{1 - x} + \frac{1}{n} \bigg)^{2\alpha + 1} \leq \lambda_n(\mu, x), \quad x \in (1-\delta,1]
\]
goes through verbatim as in the proof of Lemma \ref{perturbation_lemma}, though the upper estimate is slightly different. Let $ b > 0 $ be so large such that $ K \subseteq [-b, b] $ and let $ P_n(x) $ be extremal for $ \lambda_n(|x-1|^\alpha \chi_{[-b,1]}(x)dx, x_1) $. Define the polynomial $ R_n(x) $ as
\[
	R_n(x) = P_n(x) \bigg( 1 - \frac{(x-x_1)^{2}}{2b} \bigg)^{k n},
\]
where $ k $ is an integer yet to be determined. The degree of $ R_n $ is at most $ (2k+1)n $ and $ R_n(x_1) = 1 $. Now we have
\begin{align*}
	\lambda_{(2k+1)n}(\mu, x_1) & \leq \int |R_n(x)|^2 d\mu(x) \\
	& \leq 2w(1) \int_{1-\delta}^{1} |P_n(x)|^2 |x - 1|^\alpha dx  + \int_{K \setminus [1-\delta, 1]} |R_n(x)|^2 d\mu(x).
\end{align*}
On one hand, the extremality of $ P_n $ implies that
\[
	\int_{1-\delta}^{1} |P_n(x)|^2 |x - 1|^\alpha dx \leq \lambda_n(|x-1|^\alpha \chi_{[-b,1]}(x)dx, x_1).
\]
On the other hand, since
\[
	\sup_{x \in K \setminus [1-\delta, 1]} \bigg| 1 - \frac{(x-x_1)^{2}}{2b} \bigg|^{kn} \leq \gamma^n
\]
for some $ |\gamma| < 1 $ depending on $ k $, we have
\[
	\int_{K \setminus [1-\delta, 1]} |R_n(x)|^2 d\mu(x) \leq \mu(K) \gamma^n \| P_n \|_{[-b,b]}.
\]
Since the measure $ |x-1|^\alpha \chi_{[-b,1]}(x) dx $ is regular in the sense of Stahl-Totik, (\ref{reg_dirichlet}) implies that for all $ \tau > 0 $
\[
	\| P_n \|_{[-b,1]} \leq (1 + \tau)^n \lambda_n(|x-1|^\alpha \chi_{[-b,1]}(x)dx, x_1)^{1/2}
\]
holds if $ n $ is large enough. In addition, the Bernstein-Walsh lemma, see \cite[Theorem 5.5.7a]{Ransford}, gives that
\[
	\| P_n \|_{[-b,b]} \leq c^n \| P_n \|_{[-b,1]}
\]
holds for some, possibly very large constant $ c $. Overall, we have
\[
	\int_{K \setminus [1-\delta, 1]} |R_n(x)|^2 d\mu(x) \leq \mu(K) \gamma^n c^n (1 + \tau)^n \lambda_n(|x-1|^\alpha \chi_{[-b,1]}(x)dx, x_1)^{1/2}.
\]
If the integer $ k $ in the definition of $ R_n(x) $ is selected such that $ \gamma c (1 + \tau) < 1 $ holds (recall that $ \gamma $ depends on $ k $), the above integral is small, that is,
\[
	\int_{K \setminus [1-\delta, 1]} |R_n(x)|^2 d\mu(x) = O(q^n)
\]
for some $ |q| < 1 $. These estimates give that
\[
	\lambda_{(2k+1)n}(\mu, x_1) \leq C \lambda_n(|x-1|^\alpha \chi_{[-b,1]}(x)dx, x_1),
\]
where is $ C $ is a fixed constant. Now Lemma \ref{subsequence_lemma} yields
\[
	\lambda_n(\mu, x) \leq \frac{C}{n} \bigg( \sqrt{1 - x} + \frac{1}{n} \bigg)^{2\alpha + 1}, \quad x \in (1 - \delta, 1]
\]
for some possibly different constant $ C $. Overall, we have
\begin{equation}\label{Nevai_general_bound_edge}
	\frac{1}{Cn} \bigg( \sqrt{1 - x} + \frac{1}{n} \bigg)^{2\alpha + 1} \leq \lambda_n(\mu, x) \leq \frac{C}{n} \bigg( \sqrt{1 - x} + \frac{1}{n} \bigg)^{2\alpha + 1}, \quad x \in (1 - \delta, 1].
\end{equation}
Now define the polynomial $ Q_n(x) $ as
\[
	Q_n(x) = \frac{\lambda_n(\mu, 1 - a/n^2)}{\lambda_n(\mu, 1 - a/n^2 - x)}.
\]
The bound (\ref{Nevai_general_bound_edge}) gives that
\[
	|Q_n(x)| \leq C, \quad x \in [-a/n^2, \delta/2],
\]
holds uniformly for $ a $ in compact subsets of $ [0, \infty) $. The classical Markov inequality for $ [-a/n^2, \delta/2] $, see \cite[Chapter 4, Theorem 1.4]{DeVore-Lorentz}, implies that
\[
	|Q_n^{(k)}(x)| \leq C M^k n^{2k}, \quad x \in [-a/n^2, \delta/2],
\]
holds for some constants $ C $ and $ M $. Now we have
\[
	|Q_n(-\varepsilon_n)| \leq \sum_{k=0}^{2n-2} \frac{|Q_n^{(k)}(0)|}{k!} \varepsilon_n^k \leq 1 + C \sum_{k=1}^{2n-2} \frac{M^k n^{2k} \varepsilon_n^k}{k!} = 1 + o(1),
\]
which yields
\[
	\limsup_{n \to \infty} \frac{\lambda_n(\mu, 1 - a/n^2)}{\lambda_n(\mu, 1 - a/n^2 + \varepsilon_n)} \leq 1.
\]
To obtain the matching bound
\[
	\limsup_{n \to \infty} \frac{\lambda_n(\mu, 1 - a/n^2 + \varepsilon_n)}{\lambda_n(\mu, 1 - a/n^2)} \leq 1,
\]
define the polynomial
\[
	Q_n(x) = \frac{\lambda_n(\mu, 1 - a/n^2 + \varepsilon_n)}{\lambda_n(\mu, 1 - a/n^2 + \varepsilon_n - x)}
\]
and repeat the same argument as above to see that $ |Q_n(\varepsilon_n)| \leq 1 + o(1) $.
\end{proof}

\subsection{Proof of Theorem \ref{main_theorem_bulk}}

Throughout this section, let $ K $ be a compact set and let $ x_0 \in \operatorname{K} $ be an element in its interior. Let $ \mu $ be a measure with $ \operatorname{supp}(\mu) = K $ and suppose that $ \mu $ is absolutely continuous in a small neighbourhood of $ x_0 $ with
\[
	d\mu(x) = w(x) |x - x_0|^\alpha dx
\]
there, where $ \alpha > -1 $ and $ w(x) $ is strictly positive and continuous in $ x_0 $. 

Before we study the asymptotics of Christoffel functions with respect to $ \mu $, we prove a general lemma about the approximating properties of polynomial inverse images. Although the following lemma is known, we could not find an explicit reference for this, and its proof is scattered throughout the literature, therefore we have found it easier to include it here.

\begin{lemma}\label{bulk_approximation}
Let $ K $ be a compact set. Suppose that $ x_0 \in \operatorname{int}(K) $ is a point in its interior and let $ \varepsilon > 0 $ and $ \eta > 0 $ be arbitrary. There exists a set $ E_N = \cup_{k=0}^{N-1} [a_k,b_k] $, $ b_k \leq a_{k+1} $ such that \\
(a) $ E_N = T_{N}^{-1}([-1,1]) $, where $ T_N $ is an admissible polynomial of degree $ N $ with $ T_N(x_0) = 0 $ and $ T_N^{\prime}(x_0) > 0 $, \\
(b) $ K \subseteq E_N $, \\
(c) $ \operatorname{dist}(K, E_N) < \varepsilon $, where $ \operatorname{dist}(K, E_N) $ denotes the Hausdorff distance of $ K $ and $ E_N $,\\
(d) $ \frac{1}{1 + \eta} \omega_K(x_0) \leq \omega_{E_N}(x_0) \leq \omega_K(x_0) $, where $ \omega_S(x) $ denotes the equilibrium density of a set $ S $. \\
Moreover, we have
\begin{equation}\label{inverse_image_equilibrium}
	\omega_{E_N}(x_0) = \frac{|T_{N}^{\prime}(x_0)|}{N \pi}.
\end{equation}
\end{lemma}
\begin{proof}
Since $ K $ is a compact subset, its complement can be obtained as
\[
	\mathbb{R} \setminus K = (-\infty, a^*) \cup (b^*, \infty) \cup \Big( \bigcup_{k=0}^{\infty} (a_k^*, b_k^*) \Big),
\]
where $ a^* \leq a_k^* $ and $ b_k^* \leq b^* $ for all $ k \in \{1, 2, \dots \} $. Hence the set
\[
	F_m = \mathbb{R} \setminus \Big( (-\infty, a^*) \cup (b^*, \infty) \cup \Big( \bigcup_{k=0}^{m-1} (a_k^*, b_k^*) \Big) \Big)
\]
is a finite union of intervals. If $ m $ is large enough, $ \operatorname{dist}(K, F_m) \leq \varepsilon/2 $ and, as \cite{Totik_3} Lemma 4.2 implies,
\[
	\bigg( \frac{1}{1 + \eta} \bigg)^{1/2} \omega_{K}(x_0) \leq \omega_{F_m}(x_0) \leq \omega_K(x_0)
\]
holds. Now \cite[Theorem 2.1]{Totik_2} gives an admissible polynomial $ T_N $ and a set $ E_N = T_N^{-1}([-1,1]) = \cup_{k=0}^{N-1} [a_k, b_k] $ such that the endpoints of $ E_N $ are arbitrarily close to the endpoints of $ F_m $. Using Chebyshev polynomials $ \mathcal{T}_n(x) $ and replacing $ T_N(x) $ with $ T_N(\mathcal{T}_n(x)) $ and introducing a very small shift if necessary it can be achieved that $ T_N(x_0) = 0 $. (For details on this idea, see \cite{Totik_4}.) By multiplying with $ (-1) $ if necessary, it can also be achieved that $ T_N^{\prime}(x_0) > 0 $, therefore the conditions (a)-(c) are satisfied. If the endpoints of $ E_N $ are close enough to the endpoints of $ F_m $, then \cite[Lemma 4.2]{Totik_3} again implies that $ E_N $ satisfies the condition (d). The identity (\ref{inverse_image_equilibrium}) is a direct consequence of (\ref{inverse_image_equilibrium_general}).
\end{proof}

To prove (\ref{general_bulk}), it is enough to show that the upper and lower estimates
\[
	\limsup_{n \to \infty} \lambda_n \bigg(\mu, x_0 + \frac{a}{n} \bigg) \leq \frac{w(x_0)}{(\pi \omega_K(x_0))^{\alpha + 1}} \Big( \mathbb{L}_{\alpha}^{*}\big(\pi \omega_K(x_0) a \big) \Big)^{-1}
\]
and
\[
	\frac{w(x_0)}{(\pi \omega_K(x_0))^{\alpha + 1}} \Big( \mathbb{L}_{\alpha}^{*}\big(\pi \omega_K(x_0) a \big) \Big)^{-1} \leq \liminf_{n \to \infty} \lambda_n \bigg(\mu, x_0 + \frac{a}{n} \bigg)
\]
hold. \\

\textbf{Upper estimate.} Let $ \eta > 0 $ be arbitrary and let $ E_N = \cup_{k=0}^{N-1} [a_k, b_k] = T_N^{-1}([-1,1]) $ and $ T_N $ be the approximating set and the matching admissible polynomial granted by Lemma \ref{bulk_approximation}. (For the purpose of the upper estimate, $ \varepsilon > 0 $ in Lemma \ref{bulk_approximation} can be chosen arbitrarily. However, this will not be the case for the lower estimate.) It can be assumed without the loss of generality that $ x_0 \in (a_0, b_0) $ and $ T_N^{\prime}(x_0) > 0 $. Select a $ \delta > 0 $ so small such that
\begin{equation}\label{general_upper_separation}
\begin{aligned}
	(1) & \quad \frac{1}{1 + \eta} w(x) \leq w(x_0) \leq (1 + \eta) w(x), \\
	(2) & \quad \frac{1}{1 + \eta} |T_N(x)| \leq |T_{N}^{\prime}(x_0)| |x - x_0| \leq (1 + \eta) |T_N(x)|, \\
	(3) & \quad \frac{1}{1 + \eta} |T_N^{\prime}(x)| \leq |T_{N}^{\prime}(x_0)| \leq (1 + \eta) |T_{N}^{\prime}(x)|. \\
\end{aligned}
\end{equation}
holds for all $ [x_0 - \delta, x_0 + \delta] $. (This can be achieved since $ w $ is continuous and $ T_N $ is continuously differentiable at $ x_0 $.) Let $ \xi \in \mathbb{R} $ be arbitrary and let $ x_0 + \xi_n $ be the unique element of $ [a_0, b_0] $ such that $ T_N(x_0 + \xi_n) = \xi/n $. Since $ T_N $ is a polynomial, $ \xi_n = O(n^{-1}) $ and
\[
	\frac{\xi}{n} = T_N(x_0 + \xi_n) = T_N(x_0) + T_N^{\prime}(x_0)\xi_n + O(n^{-2}) = T_N^{\prime}(x_0)\xi_n + O(n^{-2}),
\]
which implies
\begin{equation}\label{general_upper_alpha_n}
	\xi_n = \frac{\xi}{|T_N^{\prime}(x_0)|n} + o(n^{-1}).
\end{equation}
Assume that $ P_n $ is extremal for $ \lambda_n(\mu_\alpha^b, \xi/n) $ and define
\[
	R_n(x) = P_n(T_N(x)) S_{n, x_0 + \xi_n, K}(x),
\]
where $ S_{n, x_0 + \xi_n, K}(x) $ is defined by
\begin{equation}\label{fdp}
	S_{n,x_0 + \xi_n, K}(x) = \bigg( 1 - \Big(\frac{x_0 + \xi_n - x}{\operatorname{diam}(K)}\Big)^2 \bigg)^{\lfloor \eta n \rfloor}
\end{equation}
(The reason why we choose an arbitrary $ \xi \in \mathbb{R} $ instead of the $ a $ appearing in (\ref{general_bulk}) will become apparent at the end of our calculations, where it will be clear that some scaling is necessary.) This way $ R_n $ is a polynomial of degree less than $ nN + 2\lfloor \eta n \rfloor $ with $ R_n(x_0 + \xi_n) = 1 $.
Now we have
\begin{align*}
	\lambda_{nN + 2\lfloor \eta n \rfloor}(\mu, x_0 + \xi_n) & \leq \int |R_n(x)|^2 d\mu(x) \\
	& = \int_{x_0 - \delta}^{x_0 + \delta} |R_n(x)|^2 w(x) |x - x_0|^{\alpha} dx \\
	& \phantom{aaaaaa} + \int_{K \setminus [x_0 - \delta, x_0 + \delta]} |R_n(x)|^2 d\mu(x).
\end{align*}
On one hand, (\ref{lemniscate_integral}) and (\ref{general_upper_separation}) gives
\begin{align*}
	\int_{x_0 - \delta}^{x_0 + \delta} & |R_n(x)|^2 w(x) |x - x_0|^{\alpha} dx \\
	& \leq \int_{x_0 - \delta}^{x_0 + \delta} |P_n(T_N(x))|^2 w(x) |x - x_0|^{\alpha} dx \\
	& = \int_{x_0 - \delta}^{x_0 + \delta} |P_n(T_N(x))|^2 \frac{|T_N^{\prime}(x_0)|^{\alpha + 1}}{|T_N^{\prime}(x_0)|^{\alpha + 1}} w(x) |x - x_0|^{\alpha} dx \\
	& \leq (1 + \eta)^{\alpha + 2} \frac{w(x_0)}{|T_N^{\prime}(x_0)|^{\alpha + 1}} \int_{x_0 - \delta}^{x_0 + \delta} |P_N(T_N(x))|^2 |T_N^{\prime}(x)| |T_N(x)|^{\alpha} dx \\
	& \leq (1 + \eta)^{\alpha + 2} \frac{w(x_0)}{|T_N^{\prime}(x_0)|^{\alpha + 1}} \int_{a_0}^{b_0} |P_N(T_N(x))|^2 |T_N^{\prime}(x)| |T_N(x)|^{\alpha} dx \\
	& = (1 + \eta)^{\alpha + 2} \frac{w(x_0)}{|T_N^{\prime}(x_0)|^{\alpha + 1}} \int_{-1}^{1}|P_n(x)|^2 |x|^{\alpha} dx \\
	& = (1 + \eta)^{\alpha + 2} \frac{w(x_0)}{|T_N^{\prime}(x_0)|^{\alpha + 1}} \lambda_n(\mu_\alpha^b, \xi/n).
\end{align*}
On the other hand, as implied by the Erd\H{o}s-Tur\'an criterion, $ \mu_\alpha^b $ is regular in the sense of Stahl-Totik, hence for every $ \tau > 0 $
\[
	\| P_n \|_{[-1,1]} \leq (1 + \tau)^n \| P_n \|_{L^2(\mu_\alpha^b)} \leq C (1 + \tau)^n
\]
holds for all large $ n $, where we used the extremality of $ P_n $ with respect to $ \lambda_n(\mu_\alpha^b, \xi/n) $ and (\ref{model_case_bulk_Christoffel}). The polynomial $ S_{n, x_0 + \xi_n, K}(x) $ defined by (\ref{fdp}) is decreasing exponentially fast, that is 
\[
	\| S_{n, x_0 + \xi_n, K}\|_{K \setminus [x_0 - \delta, x_0 + \delta]} \leq \gamma^n
\]
holds for some $ |\gamma| < 1 $ if $ n $ is large enough. If $ \tau $ is selected so small that $ q = (1 + \tau) \gamma < 1 $, then
\[
	\int_{K \setminus [x_0 - \delta, x_0 + \delta]} |R_n(x)|^2 d\mu(x) = O(q^n),
\]
that is, this residual integral is also decreasing exponentially fast. Combining these estimates, it follows that
\[
	\lambda_{nN + 2\lfloor \eta n \rfloor}(\mu, x_0 + \xi_n) \leq O(q^n) + (1 + \eta)^{\alpha + 2} \frac{w(x_0)}{|T_N^{\prime}(x_0)|^{\alpha + 1}} \lambda_n(\mu_\alpha^b, \xi/n).
\]
This is almost what we need. Since
\begin{equation}\label{general_upper_subsequence_and_perturbation_1}
\begin{aligned}
	\frac{\xi}{|T_N^{\prime}(x_0)|n} & = \frac{\xi N}{|T_N^{\prime}(x_0)| (nN + 2\lfloor \eta n \rfloor)} (1 + 2\eta/N) + \frac{2\xi}{|T_N^{\prime}(x_0)|} \frac{\frac{\lfloor \eta n \rfloor}{n} - \eta}{nN + 2\lfloor \eta n \rfloor} \\
	& = \frac{\xi N}{|T_N^{\prime}(x_0)| (nN + 2\lfloor \eta n \rfloor)} (1 + 2\eta/N) + o(n^{-1}),
\end{aligned}
\end{equation}
it follows from Lemma \ref{perturbation_lemma}, (\ref{general_upper_alpha_n}) and (\ref{general_upper_subsequence_and_perturbation_1}) that
\begin{equation}
\begin{aligned}
	\limsup_{n \to \infty} & (nN + 2\lfloor \eta n \rfloor)^{\alpha + 1} \lambda_{nN + 2\lfloor \eta n \rfloor}(\mu, x_0 + \xi_n) \\
	& = \limsup_{n \to \infty} (nN + 2\lfloor \eta n \rfloor)^{\alpha + 1} \\
	& \phantom{aaaaaaaaaaa} \times \lambda_{nN + 2\lfloor \eta n \rfloor}\bigg(\mu, x_0 + \frac{\xi N(1 + 2\eta/N)}{|T_N^{\prime}(x_0)| (nN + 2\lfloor \eta n \rfloor)} \bigg).
\end{aligned}
\end{equation}
If $ k $ is selected such that $ nN + 2\lfloor \eta n \rfloor \leq k \leq (n + 1)N + 2\lfloor \eta (n + 1) \rfloor $, we have
\begin{equation}
\begin{aligned}
	\bigg( & \frac{nN + 2\lfloor \eta n \rfloor}{k} \bigg)^{\alpha + 1}  k^{\alpha + 1} \lambda_k \bigg( \mu, x_0 + \frac{\xi N(1 + 2\eta/N)}{|T_N^{\prime}(x_0)| (nN + 2\lfloor \eta n \rfloor)} \bigg) \\
	& \leq (nN + 2\lfloor \eta n \rfloor)^{\alpha + 1} \lambda_{nN + 2\lfloor \eta n \rfloor} \bigg( \mu, x_0 + \frac{\xi N(1 + 2\eta/N)}{|T_N^{\prime}(x_0)| (nN + 2\lfloor \eta n \rfloor)} \bigg).
\end{aligned}
\end{equation}
Since $ (nN + 2\lfloor \eta n \rfloor)/k = 1 + o(1) $, these estimates, along with Lemma \ref{perturbation_lemma}, (\ref{model_case_bulk_Christoffel}) and (\ref{inverse_image_equilibrium}) imply
\begin{equation}
\begin{aligned}
	\limsup_{k \to \infty} & k^{\alpha + 1} \lambda_k \bigg( \mu, x_0 + \frac{\xi N(1 + 2\eta/N)}{|T_N^{\prime}(x_0)|k} \bigg) \\
	& \leq \limsup_{k \to \infty} (nN + 2\lfloor \eta n \rfloor)^{\alpha + 1} \lambda_{nN + 2\lfloor \eta n \rfloor}(\mu, x_0 + \xi_n) \\
	& \leq \limsup_{k \to \infty} \bigg(1 + \frac{2\lfloor \eta n \rfloor}{n N} \bigg)^{\alpha + 1} (1 + \eta)^{\alpha + 2} \frac{w(x_0) N^{\alpha + 1}}{|T_N^{\prime}(x_0)|^{\alpha + 1}} n^{\alpha + 1}\lambda_n(\mu_\alpha^b, \xi/n) \\
	& = (1 + 2\eta/N)^{\alpha + 1} (1 + \eta)^{\alpha + 2} \frac{w(x_0)}{(\pi \omega_{E_N}(x_0))^{\alpha + 1}} \Big( \mathbb{L}_{\alpha}^{*}(\xi) \Big)^{-1},
\end{aligned}
\end{equation}
which, by selecting $ a = \frac{\xi (1 + 2\eta/N)}{\pi \omega_{E_n}(x_0)} $, gives
\begin{equation}
\begin{aligned}
	\limsup_{k \to \infty} & k^{\alpha + 1} \lambda_k \bigg( \mu, x_0 + \frac{a}{k} \bigg) \\
	& \leq (1 + 2\eta/N)^{\alpha + 1} (1 + \eta)^{\alpha + 2} \frac{w(x_0)}{(\pi \omega_{E_N}(x_0))^{\alpha + 1}} \bigg( \mathbb{L}_{\alpha}^{*}\bigg( \frac{\pi \omega_{E_N}(x_0) a}{1 + 2\eta/N} \bigg) \bigg)^{-1}.
\end{aligned}
\end{equation}
Since $ \eta $ was arbitrary and $ \mathbb{L}_{\alpha}^{*}(\cdot) $ is continuous, we have
\begin{equation}
	\limsup_{k \to \infty} k^{\alpha + 1} \lambda_k \bigg( \mu, x_0 + \frac{a}{k} \bigg)  \leq \frac{w(x_0)}{(\pi \omega_{E_N}(x_0))^{\alpha + 1}} \Big( \mathbb{L}_{\alpha}^{*}\big(\pi \omega_{E_N}(x_0) a \big) \Big)^{-1}.
\end{equation}
The approximating set $ E_N $ was selected such that $ \omega_{E_N}(x_0) $ is arbitrarily close to $ \omega_K(x_0) $, therefore this gives us the desired upper estimate
\begin{equation}\label{general_upper_estimate}
	\limsup_{k \to \infty} k^{\alpha + 1} \lambda_k \bigg( \mu, x_0 + \frac{a}{k} \bigg)  \leq \frac{w(x_0)}{(\pi \omega_{K}(x_0))^{\alpha + 1}} \Big( \mathbb{L}_{\alpha}^{*}\big(\pi \omega_K(x_0) a \big) \Big)^{-1}.
\end{equation}
Note that since (\ref{model_case_bulk_Christoffel}) is uniform for $ a $ in compact subsets of the real line, this upper estimate is also uniform. \\

\textbf{Lower estimate for sets regular with respect to the Dirichlet problem.} For the upper estimate the Stahl-Totik regularity of $ \mu $ was not used. However, it will be needed for the lower estimate, therefore we prove it first for sets regular with respect to the Dirichlet problem to reduce technical difficulties. If a set is such, the Stahl-Totik regularity for a measure supported there gives us the uniform estimate (\ref{reg_dirichlet}). Therefore assume that $ K $ is regular with respect to the Dirichlet problem. Let $ \eta > 0 $ be arbitrary but fixed, moreover let $ \delta_1 > 0 $ so small such that
\begin{equation}\label{general_lower_dirichlet_separation_1}
	(1) \quad \frac{1}{1 + \eta} w(x_0) \leq w(x) \leq (1 + \eta) w(x_0)
\end{equation}
holds for all $ x \in [x_0 - \delta_1, x_0 + \delta_1] $. Now let $ E_N = \cup_{k=0}^{N-1} [a_k,b_k] $ be the approximating set for $ K $ and $ T_N $ be the matching admissible polynomial given by Lemma \ref{bulk_approximation}. We can assume without the loss of generality that $ x_0 \in (a_0, b_0) $. At the moment, the $ \varepsilon $ which controls the distance of $ E_N $ and $ K $ is arbitrary, but soon we'll select this parameter according to our purpose. Assume that $ P_n $ is extremal for $ \lambda_n(\mu, x_0 + a/n) $. Let
\[
	R_n(x) = P_n(x) S_{n, x_0 + a/n, E}(x),
\]
where $ S_{n,x_0 + a/n, E}(x) $ is defined similarly as in (\ref{fdp}), i.e. let $ E = [-m,m] $ be an interval so large such that $ K \subseteq [-m/2, m/2] $ and for an arbitrary $ \eta > 0 $ define
\[
	S_{n, x_0 + a/n, E}(x) = \bigg( 1 - \Big( \frac{x_0 + a/n - x}{2m}\Big)^2 \bigg)^{\lfloor \eta n \rfloor}.
\]
The large interval $ E = [-m,m] $ is needed to avoid dependence of $ S_{n, x_0+a/n, E} $ on the approximating set $ E_N $. We only need $ S_{n, x_0 + a/n, E} $ to be fast decreasing on $ E_N $, but we also want to make sure that the rate of decrease does not depend on $ E_N $, because actually the set $ E_N $ will be choosen to fit the rate of decrease of $ S_{n, x_0 + a/n, E}(x) $. 

Because $ \mu $ is regular in the sense of Stahl-Totik, (\ref{reg_dirichlet}) gives that for arbitrary $ \tau > 0 $,
\[
	\| P_n \|_{K} \leq (1+\tau)^n \| P_n \|_{L^2(\mu)}
\]
holds if $ n $ is large enough. Now the Bernstein-Walsh lemma, see \cite[Theorem 5.5.7a]{Ransford}, says that if $ E_N $ is selected accordingly (that is, the Hausdorff distance $ \operatorname{dist}(E_N,K) $ is small enough), we have
\[
	\| P_n \|_{E_N} \leq (1 + \tau)^n \| P_n \|_K
\]
Overall, since $ \sup_{x \in E_N \setminus [x_0 - \delta_1, x_0 + \delta_1]} |S_{n, x_0 + a/n, E}(x)| \leq \gamma^n $ for some $ \gamma < 1 $,
\begin{equation}\label{general_bulk_lower_R_n_1}
	\|R_n\|_{E_N \setminus [x_0 - \delta_1, x_0 + \delta_1]} \leq (1 + \tau)^{2n} \gamma^n \| P_n \|_{L^2(\mu)} \leq (1 + \tau)^{2n} \gamma^n
\end{equation}
holds, where in the final step we used the extremality of $ P_n $. Now select $ \tau $ such that $ q = (1 + \tau)^2 \gamma < 1  $. Note that this means fixing $ E_N $, because small $ \tau $ can be achieved if $ \operatorname{dist}(E_N, K) $ is small enough in Lemma \ref{bulk_approximation}.

Let $ \delta_2 > 0 $ be so small such that $ \delta_2 < \delta_1 $, moreover $ [x_0 - \delta_2, x_0 + \delta_2] \subseteq [a_0, b_0] $ and
\begin{equation}\label{general_lower_dirichlet_separation_2}
\begin{aligned}
	(2) & \quad \frac{1}{1 + \eta} |T_N^{\prime}(x_0)| \leq |T_N^{\prime}(x)| \leq (1 + \eta) |T_N^{\prime}(x_0)|, \\
	(3) & \quad \frac{1}{1 + \eta} |T_N^{\prime}(x_0)||x - x_0| \leq |T_N(x)| \leq (1 + \eta) |T_N^{\prime}(x_0)||x - x_0|
\end{aligned}
\end{equation}
holds for all $ x \in [x_0 - \delta_2, x_0 + \delta_2] $.  Since $ w(x)|x-x_0|^\alpha $ is bounded from above and below on the intervals $ [x_0 - \delta_1, x_0 - \delta_2] $ and $ [x_0 + \delta_2, x_0 + \delta_1] $, Nikolskii's inequality can be used, see \cite[Chapter 4, Theorem 2.6]{DeVore-Lorentz}, which gives
\begin{align*}
	\| P_n \|_{[x_0 - \delta_1, x_0 + \delta_1] \setminus [x_0 - \delta_2, x_0 + \delta_2]} & \leq C n \| P_n \|_{L^2(\mu)} \leq C n,
\end{align*}
for some constant $ C $, where again the extremality of $ P_n $ was used.  It follows that we have
\begin{equation}\label{general_bulk_lower_R_n_2}
	\| R_n \|_{[x_0 - \delta_1, x_0 + \delta_1] \setminus [x_0 - \delta_2, x_0 + \delta_2]} \leq C r^n n
\end{equation}
for some $ |r| < 1 $, which is dependent on $ E_N $ through $ \delta_2 $. Inside the interval $ [x_0 - \delta_2, x_0 + \delta_2] $, the Nikolskii-type inequality \cite[Lemma 2.7]{Danka-Totik} for generalized Jacobi weights can be used to obtain
\begin{equation}\label{general_bulk_lower_R_n_3}
	\| R_n \|_{[x_0 - \delta_2, x_0 + \delta_2]} \leq C n^{\max \{ 1/2, (1+\alpha)/2 \}}.
\end{equation}
For arbitrary $ y \in E_N $ we introduce the notation 
\[
	\{ y_0, y_1, \dots, y_{N-1} \} = T_N^{-1}(T_N(y)).
\]
It can be assumed without the loss of generality that $ y_k \in [a_k, b_k] $. Define
\[
	R_{n}^{*}(y) = \sum_{k=0}^{N-1} R_n(y_k).
\]
$ R_n^*$ is a polynomial of degree less than $ n + 2\lfloor \eta n \rfloor $ and there exists a polynomial $ V_n $ of degree at most $ (n + 2\lfloor \eta n \rfloor)/N $ such that 
\begin{equation}\label{R_n_star_as_polynomial_of_T_N}
	R_n^*(y) = V_n(T_N(y)).
\end{equation}
The proof of this fact can be found, for example in \cite[Section 5]{Totik_4}. For $ R_n^* $ it is also true that
\begin{equation}\label{general_lower_R_n}
\begin{aligned}
	|R_n^*(y)|^2 & = |R_n(y)|^2 + O(q^n), \quad y \in [x_0 - \delta_2, x_0 + \delta_2] \\
	|R_n^*(y)|^2 & = O(q^n), \quad y \in [a_0, b_0] \setminus [x_0 - \delta_2, x_0 + \delta_2]
\end{aligned}
\end{equation}
holds for some $ |q| < 1 $. Indeed, in general, we have
\begin{equation}\label{general_lower_R_n_absolute_value_triangle}
	|R_n^*(y)|^2 \leq \sum_{l=0}^{N-1} \sum_{k=0}^{N-1} |R_n(y_l)| |R_n(y_k)|.
\end{equation}
Because among the values $ \{ y_0, y_1, \dots, y_{N-1} \} = T_N^{-1}(T_N(y)) $ only $ y = y_0 $ is contained in $ [a_0,b_0] $, the estimates (\ref{general_bulk_lower_R_n_1}) and (\ref{general_bulk_lower_R_n_2}) gives that all terms $ |R_n(y_l)| $ with the possible exception of $ |R_n(y_0)| $ are exponentially small, which gives (\ref{general_lower_R_n}).

Now we can proceed to estimate the Christoffel functions. Let $ \alpha_n $ be the unique element in $ [-1,1] $ such that $ T_N(x_0 + a/n) = \alpha_n $. Using the Taylor expansion of $ T_N(x) $, it is clear that
\[
	\alpha_n = T_{N}^{\prime}(x_0) \frac{a}{n} + o(n^{-1}).
\]
For the polynomial $ V_n(x) $ defined in (\ref{R_n_star_as_polynomial_of_T_N}), according to (\ref{general_lower_R_n})  we have $ V_n(\alpha_n) = 1 + o(1) $, which implies 
\begin{align*}
	(1 + o(1)) \lambda_{\deg(V_n)}&(\mu_\alpha^b, \alpha_n) \leq \int_{-1}^{1} |V_n(x)|^2 |x|^\alpha dx \\
	& = \int_{a_0}^{b_0} |V_n(T_N(x))|^2 |T_N(x)|^\alpha |T_N^{\prime}(x)| dx \\
	& = \int_{x_0 - \delta_2}^{x_0 + \delta_2} |R_n^*(x)|^2 |T_N(x)|^\alpha |T_N^{\prime}(x)| dx \\
	& \phantom{aaaaa} + \bigg( \int_{a_0}^{x_0 - \delta_2} + \int_{x_0 + \delta_2}^{b_0} \bigg) |R_n^*(x)|^2 |T_N(x)|^\alpha |T_N^{\prime}(x)| dx.
\end{align*}
On one hand, using (\ref{general_lower_dirichlet_separation_1}), (\ref{general_lower_dirichlet_separation_2}) and (\ref{general_lower_R_n}) we have
\begin{align*}
	\int_{x_0 - \delta_2}^{x_0 + \delta_2} & |R_n^*(x)|^2 |T_N(x)|^\alpha |T_N^{\prime}(x)| dx \\
	& \leq (1 + \eta)^{\alpha + 2} \frac{|T_N^{\prime}(x_0)|^{\alpha + 1}}{w(x_0)} \int_{x_0 - \delta_2}^{x_0 + \delta_2} \big( O(q^n) + |R_n(x)|^2 \big) w(x) |x - x_0|^\alpha dx \\
	& \leq O(q^n) + (1 + \eta)^{\alpha + 2} \frac{|T_N^{\prime}(x_0)|^{\alpha + 1}}{w(x_0)} \int_{x_0 - \delta_2}^{x_0 + \delta_2} |P_n(x)|^2 w(x) |x - x_0|^\alpha dx  \\
	& \leq O(q^n) + (1 + \eta)^{\alpha + 2} \frac{|T_N^{\prime}(x_0)|^{\alpha + 1}}{w(x_0)} \lambda_n(\mu, x_0 + a/n).
\end{align*}
On the other hand, (\ref{general_lower_R_n}) also implies that
\[
	\bigg( \int_{a_0}^{x_0 - \delta_2} + \int_{x_0 + \delta_2}^{b_0} \bigg) |R_n^*(x)|^2 |T_N(x)|^\alpha |T_N^{\prime}(x)| dx = O(q^n),
\]
therefore the combination of these two estimates gives
\[
	(1 + o(1)) \lambda_{\deg(V_n)}(\mu_\alpha^b, \alpha_n) \leq O(q^n) + (1 + \eta)^{\alpha + 2} \frac{|T_N^{\prime}(x_0)|^{\alpha + 1}}{w(x_0)} \lambda_n(\mu, x_0 + a/n).
\]
Similarly as in (\ref{general_upper_subsequence_and_perturbation_1}) - (\ref{general_upper_estimate}), this implies the lower estimate
\[
	\frac{w(x_0)}{(\pi \omega_K(x_0))^{\alpha + 1}} \Big( \mathbb{L}_{\alpha}^{*}\big(\pi \omega_K(x_0) a \big) \Big)^{-1} \leq \liminf_{n \to \infty} \lambda_n \bigg(\mu, x_0 + \frac{a}{n} \bigg),
\]
which holds if $ K $ is regular with respect to the Dirichlet problem. Note again that since (\ref{model_case_bulk_Christoffel}) is uniform for $ a $ in compact subsets of the real line, this upper estimate is also uniform. \\

\textbf{Lower estimate for general sets.} Now we omit the assumption that $ \mathbb{C} \setminus K $ is regular with respect to the Dirichlet problem. To overcome the problem caused by this, we apply an idea from \cite{Totik_3}. For every $ \tau > 0 $ and $ m \in \mathbb{N} $ define the the set
\[
	K_{\tau, m} = \bigg\{ x \in K: \sup_{\deg(Q_n) = n} \frac{|Q_n(x)|}{\| Q_n \|_{L^2(\mu)}} \leq (1 + \tau)^n, n \geq m \bigg\}.
\]
$ F_{m, \tau} $ is compact, $ F_{m, \tau} \subset F_{m + 1, \tau} $, moreover, since $ \mu $ is regular in the sense of Stahl-Totik, we have $ \cup_{m=1}^{\infty} F_{m, \tau} = K \setminus H $, where $ H $ is a set of zero logarithmic capacity. Let $ \theta > 0 $ be arbitrary and choose $ m $ so large such that $ \operatorname{cap}(F_{m, \tau}) > \operatorname{cap}(K) - \theta/2 $. Ancona's theorem says, see \cite{Ancona}, that there is a set $ K_\theta^{*} \subseteq F_{m, \tau} $ such that $ K_{\theta}^{*} $ is regular with respect to the Dirichlet problem and 
\[
\operatorname{cap}(K_{m}^{*}) > \operatorname{cap}(F_{m, \tau}) - \theta/2 > \operatorname{cap}(K) - \theta
\]
holds. Define $ K_\theta = K_\theta^{*} \cup [x_0 - \varepsilon, x_0 + \varepsilon] $, where $ \varepsilon > 0 $ is so small such that $ \mu $ is absolutely continuous there and $ K_\theta \subseteq K $. Now $ K_\theta $ is regular with respect to the Dirichlet problem, and due to the construction of $ K_\theta $,
\[
	\frac{\| Q_n \|_{K_\theta}}{\| Q_n \|_{L^2(\mu)}} \leq (1 + \tau)^{\deg(Q_n)}
\]
holds for an arbitrary sequence of nonzero polynomials $ \{ Q_n \}_{n=1}^{\infty} $ if $ n $ is large enough. From this point, proceeding similarly as in the case of sets regular with respect to the Dirichlet problem, we obtain
\[
	\frac{w(x_0)}{(\pi \omega_{K_\theta}(x_0))^{\alpha + 1}} \Big( \mathbb{L}_{\alpha}^{*}\big(\pi \omega_{K_\theta}(x_0) a \big) \Big)^{-1} \leq \liminf_{n \to \infty} \lambda_n \bigg(\mu, x_0 + \frac{a}{n} \bigg).
\]
Since \cite[Lemma 4.2]{Totik_3} implies that $ \omega_{K_\theta}(x_0) \to \omega_K(x_0) $ as $ \theta \to 0 $, and since $ \theta > 0 $ was arbitrary, the desired lower estimate
\begin{equation}\label{general_lower_estimate}
	\frac{w(x_0)}{(\pi \omega_K(x_0))^{\alpha + 1}} \Big( \mathbb{L}_{\alpha}^{*}\big(\pi \omega_K(x_0) a \big) \Big)^{-1} \leq \liminf_{n \to \infty} \lambda_n \bigg(\mu, x_0 + \frac{a}{n} \bigg)
\end{equation}
follows. (\ref{general_upper_estimate}) and (\ref{general_lower_estimate}) gives (\ref{general_bulk}), which completes the proof of Theorem \ref{main_theorem_bulk}.
\begin{flushright}$ \Box $ \end{flushright}

\subsection{Proof of Theorem \ref{main_theorem_edge}}

Now let $ K $ be a compact subset of the real line and suppose that $ x_0 \in K $ is a right endpoint, i.e. there is an $ \varepsilon_1 > 0 $ such that $ K \cap (x_0, x_0 + \varepsilon_1) = \varnothing $. Let $ \mu $ be a measure with $ \operatorname{supp}(\mu) = K $ and suppose that $ \mu $ is absolutely continuous in $ (x_0 - \varepsilon_0, x_0] $ for some $ \varepsilon_0 > 0 $ and
\[
	d\mu(x) = w(x)|x - x_0|^\alpha dx, \quad x \in (x_0 - \varepsilon_0, x_0]
\]
there, where $ \alpha > -1 $ and $ w(x) $ is strictly positive and left-continuous in $ x_0 $. When $ x_0 $ is a right endpoint, the density of the equilibrium measure is undefined there, but a related quantity takes its place instead. The behavior of the equilibrium density $ \omega_K(x) $ at an endpoint of $ K $ can be quantified as
\[
	M(K, x_0) = \lim_{x \to x_0 -} \sqrt{2} \pi |x - x_0|^{1/2} \omega_K(x).
\]
This quantity is finite and well defined in our case. (The constant $ \sqrt{2} \pi $ is usually not incorporated in the definition of $ M(K, x_0) $, but we have found it more convenient to do so.) It has appeared several times in the literature, for example it was shown by Totik that this is the asymptotically best possible constant in Markov inequalities for polynomials in several intervals, see \cite[Theorem 4.1]{Totik_2}. The analogue of Lemma \ref{bulk_approximation} is the following.

\begin{lemma}\label{edge_approximation}
Let $ K $ be a compact subset of the real line and let $ x_0 \in K $ be a point such that $ K \cap (x_0, x_0 + \varepsilon_1) = \varnothing $ and $ [x_0 - \varepsilon_1, x_0] \subseteq K $ for some $ \varepsilon_1 > 0 $. Let $ \varepsilon > 0 $ and $ \eta > 0 $ be arbitrary. There exists a set $ E_N = \cup_{k=0}^{N-1} [a_k, b_k] $, $ b_k \leq a_{k+1} $ such that \\
(a) $ E_N = T_{N}^{-1}([-1,1]) $, where $ T_N $ is an admissible polynomial of degree $ N $, $ x_0 $ is a right endpoint of $ E_N $ with $ T_N(x_0) = 1 $ and $ T_N^{\prime}(x_0) > 0 $, \\
(b) $ K \subseteq E_N $, \\
(c) $ \operatorname{dist}(K, E_N) < \varepsilon $, where $ \operatorname{dist}(K, E_N) $ denotes the Hausdorff distance of $ K $ and $ E_N $, \\
(d) $ \frac{1}{1 + \eta} M(K, x_0) \leq M(E_N, x_0) \leq M(K, x_0) $. \\
Moreover, we have
\begin{equation}\label{M_as_T_N_derivative}
	|T_N^{\prime}(x_0)| = N^2 M(E_N, x_0)^2.
\end{equation}
\end{lemma}
\begin{proof}
The proof of (a)-(d) is almost identical to the proof of Lemma \ref{bulk_approximation}, except where we select the set $ F_m = \cup_{k=0}^{m-1} [a_k^*, b_k^*] $ which is a finite union of intervals containing $ K $, we make sure that $ x_0 $ is a right endpoint of $ F_m $. Then we select $ E_N $ using \cite[Theorem 2.1]{Totik_2}, again in such a way that $ x_0 $ remains a right endpoint of $ E_N $. Since the convergence of $ \omega_{E_N}(x) $ is locally uniform as granted by \cite[Lemma 4.2]{Totik_3}, (d) follows. The formula (\ref{M_as_T_N_derivative}) is just formula (4.10) in \cite{Totik_2}.
\end{proof}

To show (\ref{general_edge}), we shall again prove matching upper and lower estimates. In order to avoid excessive repetition, we only discuss the upper estimate, with an emphasis on the differences. The lower estimate works similarly, aside from the same differences. \\

As in the bulk, let $ \eta > 0 $ be arbitrary and let $ E_N = \cup_{k=0}^{N-1} [a_k, b_k] = T_N^{-1}([-1,1]) $ and $ T_N $ be the approximating set and the matching admissible polynomial granted by Lemma \ref{edge_approximation}. We can assume without the loss of generality that $ x_0 = b_0 $. Select a $ \delta > 0 $ so small such that
\begin{equation}\label{edge_upper_separation}
\begin{aligned}
	(1) & \quad \frac{1}{1 + \eta} w(x_0) \leq w(x) \leq (1 + \eta) w(x_0), \\
	(2) & \quad \frac{1}{1 + \eta} |T_N(x_0) - 1| \leq |T_N^{\prime}(x)||x - x_0| \leq (1 + \eta) |T_N(x_0)|, \\
	(3) & \quad \frac{1}{1 + \eta} |T_N^{\prime}(x_0)| \leq |T_N^{\prime}(x)| \leq (1 + \eta) |T_N^{\prime}(x_0)| \\
\end{aligned}
\end{equation}
holds for all $ x \in [x_0 - \delta, x_0] $. Let $ \xi \in [0, \infty) $ be arbitrary and let $ x_0 - \xi_n $ be the unique element of $ [a_0, b_0] $ such that $ T_N(x_0 - \xi_n) = 1 - \xi/(2n^2) $. Since $ T_N $ is a polynomial, we have
\[
	1 - \frac{\xi}{2n^2} = T_N(x_0 - \xi_n) = 1 - T_N^{\prime}(x_0) \xi_n + o(n^{-2}),
\]
which implies
\[
	\xi_n = \frac{\xi}{|T_N^{\prime}(x_0)| 2 n^2} + o(n^{-2}).
\]
Assume that $ P_n $ is extremal for $ \lambda_n\big(\mu_\alpha^e, 1 - \frac{\xi}{2n^2}\big) $ and define
\[
	R_n(x) = P_n(T_N(x)) S_{n, x_0 - \xi_n, K}(x),
\]
where $ S_{n, x_0 - \xi_n, K}(x) $ is defined as
\[
	S_{n,x_0 + \xi_n, K}(x) = \bigg( 1 - \Big(\frac{x_0 - \xi_n - x}{\operatorname{diam}(K)}\Big)^2 \bigg)^{\lfloor \eta n \rfloor},
\]
as usual. Then $ R_n $ is a polynomial of degree less than $ nN + 2\lfloor \eta n \rfloor $ with $ R_n(x_0 - \xi_n) = 1 $. Then, similarly as before, (\ref{edge_upper_separation}) gives
\begin{align*}
	\lambda_{nN + 2\lfloor \eta n \rfloor}(\mu, x_0 - \xi_n) & \leq \int_{x_0 - \delta}^{x_0} |R_n(x)|^2 w(x) |x - x_0|^\alpha dx \\
	& \phantom{aaaaa} + \int_{K \setminus [x_0 - \delta, x_0]} |R_n(x)|^2 w(x) |x - x_0|^\alpha dx \\
	& \leq O(q^n) + \frac{(1 + \eta)^{\alpha + 1} w(x_0)}{|T_N^{\prime}(x_0)|^{\alpha + 1}} \lambda_n\bigg(\mu_\alpha^e, 1 - \frac{\xi}{2n^2}\bigg).
\end{align*}
Now the application of Lemma \ref{subsequence_lemma} and Lemma \ref{perturbation_lemma_edge} yields that
\begin{align*}
	\limsup_{n \to \infty} & (nN + 2\lfloor \eta n \rfloor)^{2\alpha + 2} \lambda_{nN + 2\lfloor \eta n \rfloor}(\mu, x_0 - \xi_n) \\
	& = \limsup_{k \to \infty} k^{2\alpha + 2} \lambda_k \bigg( \mu, x_0 - \frac{(1 + \eta/N)^2 \xi}{2 M(E_N, x_0)^2 k^2} \bigg).
\end{align*}
which, along with the previous estimate, by selecting $ a = \frac{(1 + 2\eta/N)^2 \xi}{M(E_N, x_0)^2} $ implies
\begin{align*}
	\limsup_{k \to \infty} & k^{2\alpha + 2} \lambda_k \bigg( \mu, x_0 - \frac{a}{2k^2} \bigg) \\
	& \leq \frac{(1 + \eta)^{\alpha + 1} (1 + 2\eta/N)^{2\alpha + 2} w(x_0)}{M(E_N, x_0)^{2\alpha + 2}} \Bigg( 2^{\alpha + 1} \mathbb{J}_\alpha^*\bigg(\frac{M(E_N, x_0)^2}{(1+2\eta/N)^2} a \bigg) \Bigg)^{-1}.
\end{align*}
Since $ \eta $ was arbitrary and $ E_N $ was choosen such that Lemma \ref{bulk_approximation} (d) holds, this implies the desired upper estimate
\begin{equation}\label{edge_upper_estimate}
	\limsup_{k \to \infty} k^{2\alpha + 2} \lambda_k \bigg( \mu, x_0 - \frac{a}{2k^2} \bigg) \leq \frac{w(x_0)}{M(K, x_0)^{2\alpha + 2}} \Big( 2^{\alpha + 1} \mathbb{J}_\alpha^*\big( M(K, x_0)^2 a\big) \Big)^{-1}.
\end{equation}
The lower estimate
\begin{equation}\label{edge_lower_estimate}
	\frac{w(x_0)}{M(K, x_0)^{2\alpha + 2}} \Big( 2^{\alpha + 1} \mathbb{J}_\alpha^*\big( M(K, x_0)^2 a\big) \Big)^{-1} \leq \liminf_{k \to \infty} k^{2\alpha + 2} \lambda_k \bigg( \mu, x_0 - \frac{a}{2k^2} \bigg)
\end{equation}
can be obtained as we did in Theorem \ref{main_theorem_bulk}, except of course with the same differences which also appeared at the upper estimate as well. Finally, (\ref{edge_upper_estimate}) and (\ref{edge_lower_estimate}) gives (\ref{general_edge}), and this is what we had to prove.
\begin{flushright}$ \Box $ \end{flushright}

\section{Universality limits}\label{section_universality}

Our aim in this section is to prove Theorems \ref{main_theorem_universality_bulk} and \ref{main_theorem_universality_edge}. Theorem \ref{main_theorem_universality_edge} is a direct corollary of Theorem \ref{main_theorem_edge} using the result \cite[Theorem 1.2]{Lubinsky_3}. To prove Theorem \ref{main_theorem_universality_bulk}, we employ the second method of Lubinsky which is based upon the theory of entire functions of exponential type. We say that an entire function $ g(z) $ is of \emph{order}$ \rho $ if
\[
	\rho = \limsup_{r \to \infty} \frac{\log\big( \log \big( \sup_{|z|=r} |g(z)| \big)\big)}{\log r}.
\]
An entire function of order $ 1 $ is said to be of the \emph{exponential type} $ \sigma $ if
\[
	\sigma = \limsup_{r \to \infty} \frac{\sup_{|z|=r} \log |g(z)|}{r}.
\]
If $ g(z) $ is of the exponential type, it belongs to the \emph{Cartwright class} if
\[
	\int_{-\infty}^{\infty} \frac{\log^+ |g(x)|}{1 + x^2} dx < \infty.
\]
A sequence of entire functions $ \{ g_n(z) \}_{n=1}^{\infty} $ is said to be \emph{normal}, if every subsequence contains a subsequence which converges uniformly on compact subsets of the complex plane. It is known, see \cite[Theorem 14.6]{BigRudin}, that if $ \{ g_n(z) \}_{n=1}^{\infty} $ is uniformly bounded on each compact subset of the complex plane, then it is normal. \\

During this section we follow the lines of \cite{Lubinsky_3}. First we develop reproducing identities for the kernel function $ \mathbb{L}_\alpha^* $, then we use the theory of entire functions of exponential type to deduce universality limits from Theorem \ref{main_theorem_bulk}.

\subsection{Reproducing kernel identities for $ \mathbb{L}_\alpha^* $}

\begin{theorem}\label{reproducing_kernel_identity}
Let $ g $ be an entire function of exponential type $ 1 $ and suppose that $ |x|^{\alpha/2} g(x) \in L^2(\mathbb{R}) $ for some $ \alpha > -1 $. Then
\begin{equation}\label{reproducing_identity_equation}
	g(z) = \int_{-\infty}^{\infty} g(s) \mathbb{L}_{\alpha}^{*}(z,s) |s|^\alpha ds
\end{equation}
holds for all $ z \in \mathbb{C} $.
\end{theorem}
The proof of Theorem \ref{reproducing_kernel_identity}, given in the next lemma, is almost verbatim to the proof of \cite[Theorem 6.1]{Lubinsky_3}, therefore we shall not carry it out in detail. It depends on Lemma \ref{reproducing_kernel_lemma}, which is an analogue of \cite[Lemma 6.2]{Lubinsky_3}.
\begin{lemma}\label{reproducing_kernel_lemma} Let $ \alpha > -1 $. \\
(a) For all $ a, b \in \mathbb{R} $ we have
\begin{equation}\label{reproducing_identity_L}
	\mathbb{L}_\alpha^*(a,b) = \int_{- \infty}^{\infty} \mathbb{L}_{\alpha}^{*}(a,s) \mathbb{L}_{\alpha}^{*}(s,b) |s|^\alpha ds.
\end{equation}
(b) If $ \{ j_{\alpha, k} \}_{k=-\infty}^{\infty} $ denotes the zeros of $ J_\alpha(x) $ with $ j_{\alpha, k} = 0 $ (which is not necessarily a zero of $ J_\alpha(z) $), then
\[
	\int_{-\infty}^{\infty} \mathbb{L}_{\alpha}^{*}(j_{\frac{\alpha - 1}{2}, k}, x) \mathbb{L}_{\alpha}^{*}(x, j_{\frac{\alpha - 1}{2}, l}) |x|^{\alpha} dx = \delta_{k,l} \mathbb{L}_{\alpha}^{*}(j_{\frac{\alpha - 1}{2}, k}, j_{\frac{\alpha - 1}{2}, l}).
\]
(c) Let $ \{ c_k \}_{k=-\infty}^{\infty} \in l^2(\mathbb{Z}) $. Then
\[
	\int_{-\infty}^{\infty} \Bigg( \sum_{\substack{k=-\infty \\ k \neq 0}}^{\infty} c_k \frac{\mathbb{L}_{\alpha}^{*}(j_{\frac{\alpha - 1}{2}, k}, x)}{\sqrt{\mathbb{L}_{\alpha}^{*}(j_{\frac{\alpha - 1}{2}, k}, j_{\frac{\alpha - 1}{2}, k})}} \Bigg)^2 |x|^\alpha dx = \sum_{\substack{k = -\infty \\ k \neq 0}}^{\infty} c_k^2.
\]
(d) Let $ g $ be an entire function of exponential type $ 1 $. If $ |x|^{\alpha/2}g(x) \in L^2(\mathbb{R}) $, then
\[
	g(z) = \sum_{\substack{k=-\infty \\ k \neq 0}}^{\infty} g(j_{\frac{\alpha - 1}{2},k}) \frac{\mathbb{L}_{\alpha}^{*}(j_{\frac{\alpha - 1}{2}, k}, z)}{\mathbb{L}_{\alpha}^{*}(j_{\frac{\alpha - 1}{2}, k}, j_{\frac{\alpha - 1}{2}, k})}
\]
holds for all $ z \in \mathbb{C} $, and the series converge uniformly on compact sets.
\end{lemma}
\begin{proof}
(a) This proof was kindly provided to us by D. S. Lubinsky \cite{Lubinsky_4}. Using the reproducing kernel relations for $ K_n(\mu_\alpha^b, x, y) $, where $ \mu_\alpha^b $ is defined by (\ref{mu_alpha_bulk}), we have
\[
	K_n(\mu_\alpha^b, a/n, b/n) = \int_{-1}^{1} K_n(\mu_\alpha^b, a/n, x) K_n(\mu_\alpha^b, x, b/n) |x|^\alpha dx.
\]
Substituting $ x = s/n $, the asymptotic formula (\ref{model_case_bulk}) implies
\begin{align*}
	\mathbb{L}_{\alpha}^{*}& (a,b)  = \frac{1}{n^{\alpha + 1}} \int_{-1}^{1} K_n(\mu_\alpha^b, a/n, x) K_n(\mu_\alpha^b, x, b/n) |x|^\alpha dx + o(1) \\
	& = \frac{1}{n^{\alpha + 1}} \Big( \int_{-1}^{-r/n} + \int_{-r/n}^{r/n} + \int_{r/n}^{1} \Big) K_n(\mu_\alpha^b, a/n, x) K_n(\mu_\alpha^b, x, b/n) |x|^\alpha dx + o(1) \\
	& = \int_{-r}^{r} \mathbb{L}_\alpha^*(a,s) \mathbb{L}_\alpha^*(s,b) ds \\
	& \phantom{aaa} + \frac{1}{n^{\alpha + 1}}\Big( \int_{-1}^{-r/n} + \int_{r/n}^{1} \Big) K_n(\mu_\alpha^b, a/n, x) K_n(\mu_\alpha^b, x, b/n) |x|^\alpha dx + o(1).
\end{align*}
We will show that the last integrals are small in terms of $ n $ and $ r $. To do this, we shall use Pollard's decomposition of the Christoffel-Darboux kernel. According to \cite[(4.6)-(4.8)]{Xu} and the formula after, we have
\[
	K_n(\mu_\alpha^b, x, y) = K_{n,1}(\mu_\alpha^b, x, y) + K_{n,2}(\mu_\alpha^b, x, y) + K_{n,3}(\mu_\alpha^b, x, y),
\]
where
\begin{align*}
	K_{n,1}(\mu_\alpha^b, x, y) & = a_n p_n(x) p_n(y), \\
	K_{n,2}(\mu_\alpha^b, x, y) & = b_n \frac{(1 - y^2) p_n(x) q_{n-1}(y)}{x-y}, \\
	K_{n,3}(\mu_\alpha^b, x, y) & = b_n \frac{(1 - x^2) p_n(y) q_{n-1}(x)}{y - x},
\end{align*}
where $ a_n, b_n $ are bounded constants, $ p_n(x) $ is the $ n $-th orthonormal polynomial with respect to the measure $ |x|^\alpha dx $ and $ q_n(x) $ is the $ n $-th orthonormal polynomial with respect to the measure $ (1 - x^2)|x|^\alpha dx $. Using \cite[Lemma 29, p. 170]{Nevai_2}, we obtain the estimates
\begin{align*}
	p_n(x)^2 & \leq (\sqrt{1-x} + 1/n)^{-1} (\sqrt{1+x} + 1/n)^{-1} (|x| + 1/n)^{-\alpha}, \\
	q_n(x)^2 & \leq (\sqrt{1-x} + 1/n)^{-3} (\sqrt{1+x} + 1/n)^{-3} (|x| + 1/n)^{-\alpha}.
\end{align*}
Now the Cauchy-Schwarz inequality gives
\begin{align*}
	\frac{1}{n^{\alpha + 1}} & \int_{r/n}^{1} K_n(\mu_\alpha^b, a/n, x) K_n(\mu_\alpha^b, x, b/n) |x|^\alpha dx \\
	& \leq \Bigg( \frac{1}{n^{\alpha + 1}} \int_{r/n}^{1} K_n(\mu_\alpha^b, a/n, x)^2 |x|^\alpha dx \Bigg)^{1/2} \\ & \phantom{aaaaa} \times \Bigg( \frac{1}{n^{\alpha + 1}} \int_{r/n}^{1} K_n(\mu_\alpha^b, x, b/n)^2 |x|^\alpha dx \Bigg)^{1/2}.
\end{align*}
Suppose that $ r > \max\{a,b\} $. We have the following estimates. (In the following calculations the constant $ c $ often varies from line to line.) \\
\begin{equation}
\begin{aligned}
	\frac{1}{n^{\alpha + 1}} \int_{r/n}^{1} & K_{n,1}(\mu_\alpha^b, a/n, x)^2 |x|^\alpha dx \\
	& \leq \frac{c}{n^{\alpha + 1}} \int_{r/n}^{1} |p_n(a/n)|^2 |p_n(x)|^2 |x|^\alpha dx \\
	& \leq \frac{c}{n} \int_{r/n}^{1} |p_n(x)|^2 |x|^\alpha dx \\
	& \leq \frac{c}{n}
\end{aligned}
\end{equation}
\begin{equation}
\begin{aligned}
	\frac{1}{n^{\alpha + 1}} \int_{r/n}^{1} & K_{n,2}(\mu_\alpha^b, a/n, x)^2 |x|^\alpha dx \\
	& \leq \frac{c}{n^{\alpha + 1}} \int_{r/n}^{1}  |p_n(a/n)|^2 \bigg| \frac{(1-x^2) q_{n-1}(x)}{x - a/n} \bigg|^2 |x|^\alpha dx \\
	& \leq \frac{c}{n} \int_{r/n}^{1/2} |x|^{-2} dx + \frac{c}{n} \int_{1/2}^{1}|1 - x|^{1/2} dx \\
	& \leq \frac{c}{r} + \frac{c}{n}
\end{aligned}
\end{equation}
\begin{equation}
\begin{aligned}
	\frac{1}{n^{\alpha + 1}} \int_{r/n}^{1} & K_{n,3}(\mu_\alpha^b, a/n, x) dx \\
	& \leq \frac{c}{n^{\alpha + 1}} \int_{r/n}^{1} |1 - a^2/n^2|^2 |q_{n-1}(a/n)|^2 \bigg| \frac{p_n(x)}{x - a/n} \bigg|^2 |x|^\alpha dx \\
	& \leq \frac{c}{n} \int_{r/n}^{1/2} |x|^{-2} dx + \frac{c}{n} \int_{1/2}^{1} |1 - x|^{-1/2} dx \\
	& \leq \frac{c}{r} + \frac{c}{n}
\end{aligned}
\end{equation}
These altogether give that
\[
	\frac{1}{n^{\alpha + 1}} \int_{r/n}^{1} K_n(\mu_\alpha^b, a/n, x) K_n(\mu_\alpha^b, x, b/n) |x|^\alpha dx \leq c \Big( \frac{1}{n} + \frac{1}{r} \Big).
\]
Overall, we have
\[
	\mathbb{L}_\alpha^*(a,b) = \int_{-r}^{r} \mathbb{L}_\alpha^*(a,s) \mathbb{L}_\alpha^*(s,b) |s|^\alpha ds + O(1) \Big( \frac{1}{n} + \frac{1}{r} \Big),
\]
from which (\ref{reproducing_identity_L}) follows by letting first $ n $ then $ r $ to infinity. \\
(b) is a simple consequence of (a) and the proofs of (c)-(d) go through verbatim as in \cite[Lemma 6.2]{Lubinsky_3}.
\end{proof}

\subsection{Limits of $ K_n $}
From now on, $ K_n(z,w) $ will always denote the $ n $-th Christoffel-Darboux kernel with respect to the measure $ \mu $ in Theorem \ref{main_theorem_bulk}. Define
\begin{equation}\label{f_n_definition}
	f_n(a,b) = \frac{K_n\big(x_0 + \frac{a}{\pi \omega_K(x_0) n}, x_0 + \frac{b}{ \pi \omega_K(x_0) n} \big)}{K_n(x_0, x_0)}, \quad a, b \in \mathbb{C}.
\end{equation}
For convenience we shall use the notation
\begin{equation}\label{star_notation}
	z^* = \frac{z}{\pi \omega_K(x_0)} 
\end{equation}
for all $ z $ in the complex plane, so this way $ f_n(a,b) $ takes the form
\[
	f_n(a,b) = \frac{K_n(x_0 + a^*/n, x_0 + b^*/n)}{K_n(x_0, x_0)}.
\]
First shall prove that $ \{ f_n(a,b) \}_{n=1}^{\infty} $ is a normal family of entire functions in both variable and then we will study its possible limits.
\begin{lemma}\label{normality_lemma}
For all $ a, b \in \mathbb{C} $, we have
\begin{equation}\label{normality_f_n_estimate}
	|f_n(a,b)| \leq c_1 e^{c_2 (|\operatorname{Im}(a)| + |\operatorname{Im}(b)|)} (1 + |\operatorname{Re}(a)|)^{-\alpha/2} (1 + |\operatorname{Re}(b)|)^{-\alpha/2}
\end{equation}
for some positive constants $ c_1, c_2 $. In particular, $ \{ f_n(a,b) \}_{n=1}^{\infty} $ is a normal family of functions for $ a, b $ in compact subsets of the complex plane. 
\end{lemma}
\begin{proof}
Since by definition (note that the complex conjugate has been left off for purpose) we have
\[
	K_n(z,w) = \sum_{k=0}^{n-1} p_k(z) p_k(w),
\]
the orthonormality of the $ p_k $-s imply
\[
	K_n(z,w) = \int K_n(z,x) K_n(x, w) d\mu(x).
\]
Using the Cauchy-Schwarz inequality, we have
\begin{equation}\label{normality_estimate_1}
	\bigg|K_n\Big(x_0 + \frac{a^*}{n}, x_0 + \frac{b^*}{n} \Big)\bigg|^2 \leq \bigg|K_n\Big(x_0 + \frac{a^*}{n}, x_0 + \frac{a^*}{n} \Big)\bigg| \bigg|K_n\Big(x_0 + \frac{b^*}{n}, x_0 + \frac{b^*}{n} \Big)\bigg|
\end{equation}
Similarly as in the proof of Lemma \ref{perturbation_lemma}, define the polynomial
\[
	P_n(z) = \frac{\lambda_n(\mu, x_0 + \xi/n)}{\lambda_n(\mu, x_0 + \xi/n  + z)}
\]
for an arbitrary $ \xi \in \mathbb{R} $. As (\ref{Christoffel_polynomial_form}) implies, $ P_n $ is indeed a polynomial with $ P_n(0) = 1 $. Therefore for an arbitrary $ \eta \in \mathbb{R} $ we have
\[
	P_n(i  \eta / n) = 1 + \sum_{k=1}^{2n-2} \frac{P_{n}^{(k)}(0)}{k!} \Big( i \frac{\eta}{n} \Big)^k.
\]
According to (\ref{lambda_polynomial_bound}), $ |P_{n}^{(k)}(0)| \leq C M^k n^k $ for some constants $ C $ and $ M $, therefore
\begin{align*}
	|P_n(i \eta / n)| & \leq 1 + \sum_{k=1}^{2n-2} \frac{\big|P_{n}^{(k)}(0) \big|}{k!} \bigg| \frac{\eta}{n} \bigg|^k \leq 1 + \sum_{k=0}^{2n-2} C \frac{M^k |\eta|^k}{k!} \leq C e^{M |\eta|}.
\end{align*}
Together with this and (\ref{normality_estimate_1}), we have
\begin{align*}
	\Big|K_n\Big(x_0 + \frac{a^*}{n}, x_0 + \frac{b^*}{n} \Big)\Big|^2 & \leq \bigg|K_n\Big(x_0 + \frac{a^*}{n}, x_0 + \frac{a^*}{n} \Big)\bigg| \bigg|K_n\Big(x_0 + \frac{b^*}{n}, x_0 + \frac{b^*}{n} \Big)\bigg| \\
	& \leq C e^{M |\operatorname{Im}(a^*)|}K_n\bigg(x_0 + \frac{\operatorname{Re}(a^*)}{n}, x_0 + \frac{\operatorname{Re}(a^*)}{n} \bigg) \\
	& \phantom{aa} \times e^{M |\operatorname{Im}(b^*)|}K_n\bigg(x_0 + \frac{\operatorname{Re}(b^*)}{n}, x_0 + \frac{\operatorname{Re}(b^*)}{n} \bigg).
\end{align*}
Now, Theorem \ref{main_theorem_bulk} says that
\[
	\frac{K_n\big(x_0 + \frac{\xi}{n}, x_0 + \frac{\xi}{n} \big)}{K_n(x_0, x_0)} = (1 + o(1)) \mathbb{L}_{\alpha}^{*}(\pi \omega_K(x_0) \xi)
\]
uniformly for $ \xi $ in compact subsets of the real line. Using the formulas \cite[9.1.27]{Abramowitz-Stegun}, we have
\begin{align*}
	\mathbb{L}_{\alpha}^{*}(a) & = \frac{1}{2 a^{\alpha-1}} \Big( J_{\frac{\alpha+1}{2}}^{\prime}(a) J_{\frac{\alpha-1}{2}}(a) - J_{\frac{\alpha-1}{2}}^{\prime}(a) J_{\frac{\alpha+1}{2}}(a) \Big) \\
	& = \frac{1}{2 a^{\alpha - 1}} \Big( \Big( J_{\frac{\alpha-1}{2}}(a) - \frac{\alpha+1}{2a} J_{\frac{\alpha+1}{2}}(a) \Big) J_{\frac{\alpha - 1}{2}}(a) \\
	& \phantom{aaaaaaaaaaaa} - \Big( -J_{\frac{\alpha+1}{2}}(a) + \frac{\alpha-1}{2a} J_{\frac{\alpha-1}{2}}(a) \Big) J_{\frac{\alpha+1}{2}}(a) \Big).
\end{align*}
With this and some elementary trigonometric identities, \cite[9.2.1]{Abramowitz-Stegun} gives that for large $ \xi $ we have
\[
	J_{\frac{\alpha+1}{2}}(a) = \bigg( \frac{2}{\pi a} \bigg)^{1/2} \Big( \cos\Big(a - \frac{(\alpha + 1) \pi}{4} - \frac{\pi}{4} \Big) + O(1/a) \Big)
\]
and
\[
	J_{\frac{\alpha-1}{2}}(a) = \bigg( \frac{2}{\pi a} \bigg)^{1/2} \Big( \sin\Big(a - \frac{(\alpha + 1) \pi}{4} - \frac{\pi}{4} \Big) + O(1/a) \Big),
\]
which yields that
\begin{equation}\label{L_alpha_large_argument}
	\mathbb{L}_\alpha^*(a) = \frac{1}{\pi |a|^{\alpha}}(1 + o(1)),
\end{equation}
therefore $ \mathbb{L}_{\alpha}^{*}(\pi \omega_K(x_0) \xi) \leq C |\xi|^{-\alpha} $ holds for large $ \xi $. Since $ \mathbb{L}_{\alpha}^{*}(\pi \omega_K(x_0) \xi) \leq C (1 + |\xi|)^{-\alpha} $ obviously holds in any bounded set containing $ 0 $ for some constant, we have
\begin{align*}
\Bigg|\frac{K_n\big(x_0 + \frac{a^*}{n}, x_0 + \frac{b^*}{n} \big)}{K_n(x_0, x_0)}\Bigg|^2 & \leq C e^{c(|\operatorname{Im}(a)| + |\operatorname{Im}(b)|)}(1 + |\operatorname{Re}(a)|)^{-\alpha} (1 + |\operatorname{Re}(b)|)^{-\alpha}
\end{align*}
for some constants $ c, C $ and this implies (\ref{normality_f_n_estimate}).
\end{proof}

Now we study the possible limits of $ \{ f_n(a,b) \}_{n=1}^{\infty} $. In the next lemmas we prove that a limit of its subsequence is an entire function of exponential type belonging to the Cartwright class and we take a look at its zeros. The exponential type and the behavior of the zeros are connected, because if $ g(z) $ is an entire function of exponential type $ \sigma $ belonging to the Cartwright class, then
\[
	\lim_{r \to \infty} \frac{n(g, r)}{2r} = \frac{\sigma}{\pi}
\]
holds, where $ n(g,r) $ is the number of zeros of $ g $ in a ball of radius $ r $ centered at zero. (See \cite[Theorem 17.2.1]{Levin} for details.) Before we state our next lemma, we fix some notations about the zeros of some frequently used functions. \\

First define the function
\begin{equation}\label{psi_n_definition}
	\psi_n(z,w) = p_n(z)p_{n - 1}(w) - p_n(w) p_{n-1}(z).
\end{equation}
For real $ \xi $, the zeros of $ \psi_n(\xi, \cdot) $ will be denoted as
\begin{equation}\label{psi_n_zeros}
	\dots < t_{-1,n}(\xi) < t_{0n}(\xi) = \xi < t_{1n}(\xi) \dots.
\end{equation}
Note that these zeros are indeed real, see \cite[Theorem 3.1]{Freud}, and they are centered around $ \xi $, moreover $ t_{0n}(\xi) = \xi $ is indeed a zero of $ \psi_n(\xi, \cdot) $. The zeros of $ K_n(x_0 + a/n, \cdot) $ are denoted as
\begin{equation}\label{K_n_zeros}
	\dots < x_{-1,n}(a) < x_0 + \frac{a}{n} < x_{1n}(a) < \dots.
\end{equation}
For convenience, we write $ x_{0n}(a) = x_0 + a/n $. Note again that since $ K_n(\xi, \xi) $ is strictly positive, $ x_{0n}(a) $ cannot be a zero of $ K_n(x_0 + a/n, \cdot) $. The Christoffel-Darboux formula (\ref{Christoffel-Darboux}) says that
\[
	K_n(x,y) = \frac{\gamma_{n-1}}{\gamma_n} \frac{\psi_n(x,y)}{x - y},
\]
therefore
\begin{equation}\label{zeros_of_K_n_as_zeros_of_psi_n}
	x_{kn}(a) = t_{kn}(x_0 + a/n)
\end{equation}
holds for all integer $ k $ for which the above expression makes sense. In our case, the zeros of $ f_n(a, \cdot) $ are also important, thus they will be denoted as
\[
	\dots < \rho_{-1,n}(a) < a < \rho_{1n}(a) < \dots,
\]
and again we write $ \rho_{0n}(a) = a $ for convenience. Since $ f_n(a,a) $ is strictly positive, $ \rho_{0n}(a) $ cannot be a zero of $ f_n(a, \cdot) $. The definition of $ f_n(a, b) $ implies that
\begin{equation}\label{zeros_of_f_n_as_zeros_of_K_n}
	\rho_{kn}(a) = n\pi \omega_K(x_0)(x_{kn}(a^*) - x_0)
\end{equation}
holds for all $ k $ integers for which the above expression makes sense, where $ a^* = a/(\pi \omega_K(x_0)) $.

\begin{lemma}\label{zeros}
Let $ f(a,b) = \lim_{k \to \infty} f_{n_k}(a,b) $ for some subsequence $ n_k $. \\
(a) If $ a \in \mathbb{R} $, then all the zeros of $ f(a, \cdot) $ are real. Moreover, if $ n(f(a, \cdot), r) $ denotes the number of zeros of $ f(a, \cdot) $ in the ball of center 0 with radius $ r $, then
\begin{equation}\label{zeros_in_a_ball}
	n(f(a, \cdot), r) -  n(f(0, \cdot), r) = O(1).
\end{equation}
(b) Let
\[
	\dots \leq \rho_{-2} \leq \rho_{-1} < 0 <  \rho_1 \leq \rho_2 \leq \dots
\]
denote the zeros of $ f(0, \cdot) $ ordered around zero and write $ \rho_0 = 0 $ for convenience.  Then
\begin{equation}\label{zeros_converge}
	\rho_{kn}(0) \to \rho_k, \quad n \to \infty
\end{equation}
holds for all $ k \geq 0 $ and there are positive constants $ c_1, c_2 $ such that
\begin{equation}\label{zero_spacing}
\begin{aligned}
	\rho_{k} - \rho_{k-1} & \leq c_1, \\
	\rho_{k} - \rho_{k - 2} & \geq c_2.
\end{aligned}
\end{equation}
In particular, the zeros of $ f(0, \cdot) $ are at most double.
\end{lemma}
\begin{proof}
(a) The Christoffel-Darboux formula (\ref{Christoffel-Darboux}) gives that
\[
	f_n(a,z) = \frac{\pi \omega_K(x_0) n}{K_n(x_0, x_0)} \frac{\gamma_{n-1}}{\gamma_n} \frac{\psi_n(x_0 + \frac{a^*}{n}, x_0 + \frac{z^*}{n})}{a - z},
\]
where $ \psi_n(z,w) $ is defined by (\ref{psi_n_definition}).
A classic and well-known fact says that for real $ \xi $, all zeros of $ \psi_n(\psi, \cdot) $ are real, see for example \cite[Theorem 3.1]{Freud}. Hence Hurwitz's theorem implies that the zeros of $ f(a, \cdot) $ are also real for all $ a \in \mathbb{R} $. The proof of (\ref{zeros_in_a_ball}) goes exactly as in \cite[Lemma 4.3]{Lubinsky_3}, which we include for completeness. It is known that if $ x_{1n} < x_{2n} < \dots < x_{nn} $ denotes the zeros of the orthonormal polynomial $ p_n $, then if $ p_n(\xi) p_{n-1}(\xi) \neq 0 $, the function $ \psi_n(\xi, \cdot) $ has a simple zero in each of the intervals
\[
	(x_{1n}, x_{2n}), \dots, (x_{n-1,n}, x_{nn})
\]
plus one zero outside $ [x_{1n}, x_{nn}] $, and if $ p_n(\xi) p_{n-1}(\xi) = 0 $, then $ \psi_n(\xi, \cdot) $ is a multiple of $ p_{n-1} $ or $ p_n $, hence the interlacing property of the zeros of orthogonal polynomials imply that in the former case $ \psi_n(\xi, \cdot) $ has a zero in each of the intervals
\[
	(x_{1n}, x_{2n}), \dots, (x_{n-1,n}, x_{nn}),
\]
and in the latter case the zeros of $ \psi_n(\xi, \cdot) $ coincide with the zeros of $ p_n $. For these facts, see Theorem 2.3 and the proof of Theorem 3.1 in \cite{Freud}. Therefore, if $ n(\psi_n(a, \cdot), [c,d]) $ denotes the zeros of $ \psi_n(a, \cdot) $ in the interval $ [c,d] $, then
\[
	|n(\psi_n(a, \cdot), [x_{mn}, x_{kn}]) - (m - k)| \leq 1.
\]
Now, if $ \{ x_{jn}(a) \} $ denotes the zeros of $ K_n(x_0 + a/n, \cdot) $ centered around $ x_0 + a/n $ as in (\ref{K_n_zeros}), then $ \rho_{kn}(a) = n \pi \omega_K(x_0) (x_{kn}(a^*) - x_0) $, where $ a^* = a/(\pi \omega_K(x_0)) $. (Recall that the definition of $ f_n(a,b) $ included the scaling constant $ \pi \omega_K(x_0) $.) This, together with the previous observations about the location of the zeros of $ \psi_n(a, \cdot) $, means that if $ r $ is fixed and $ n $ is large, 
\[
	|n(f_n(a, \cdot),r) - n(f_n(0, \cdot),r)| \leq M
\]
for some constant $ M $. Hurwitz's theorem implies again that the above holds for $ f $, therefore we have (\ref{zeros_in_a_ball}). \\

(b) First note that $ \rho_0 $ can never be a zero of $ f(0, \cdot) $, since $ f_n(0,0) = 1 $ for all $ n $. Now (\ref{zeros_converge}) is immediate from Hurwitz's theorem. Since $ \mu $ is a doubling measure in a small neighborhood $ (x_0 - \varepsilon_0, x_0 + \varepsilon_0) $ of $ x_0 $ (note that $ d\mu(x) = w(x) |x - x_0|^\alpha $ there for a continuous and positive $ w $), \cite[Theorem 1.1]{Varga} says that if $ x_{kn}, x_{k+1,n}, \dots, x_{l,n} \in (x_0 - \varepsilon_0, x_0 + \varepsilon_0) $, then
\begin{equation}\label{zeros_doubling_measure_spacing}
	\frac{c}{n} \leq x_{m + 1,n} - x_{mn} \leq \frac{C}{n}, \quad m = k , k + 1, \dots, l - 1
\end{equation}
holds for some constants $ c $ and $ C $ independent of $ m $ and $ n $. Together with (\ref{zeros_doubling_measure_spacing}) and the above observation about the location of the zeros of $ f_n(a, \cdot) $, we have
\[
	x_{kn}(0) - x_{k-1,n}(0) \leq \frac{C}{n}
\]
and
\[
	x_{kn}(0) - x_{k-2,n}(0) \geq \frac{c}{n}
\]
for some possibly different constants, therefore, since we have $ \rho_{kn}(0) = n\pi \omega_K(x_0)(x_{kn}(0) - x_0) $, using Hurwitz's theorem once more gives (\ref{zero_spacing}).
\end{proof}
 
\begin{lemma}\label{normality_f_lemma}
Let $ f(a,b) = \lim_{k \to \infty} f_{n_k}(a,b) $ for some subsequence $ n_k $. \\
(a) $ f(a, \cdot) $ is entire of exponential type $ \sigma_a $ and
\begin{equation}\label{normality_f_L2_bound}
	\int_{-\infty}^{\infty} |f(a,t)|^2 |t|^\alpha dx \leq \frac{f(a, \overline{a})}{\mathbb{L}_{\alpha}^{*}(0)}
\end{equation}
holds. \\
(b) $ f(a, \cdot) $ belongs to the Cartwright class. \\
(c) The exponential type $ \sigma_a $ of the entire function $ f(a, \cdot) $ is independent of $ a $.
\end{lemma}
\begin{proof}
(a) It is clear that $ f(a,b) $ is entire in both variables, since it is a locally uniform limit of entire functions. Moreover, the bound (\ref{normality_f_n_estimate}) holds for $ f(a,b) $ as well, which implies that $ f(a, \cdot) $ is of exponential type. We shall denote its exponential type with $ \sigma_a $ for the time being. (In fact, we shall show later that the type is independent of $ a $ and it is $ 1 $.) As for the proof of (\ref{normality_f_L2_bound}), we proceed similarly as in \cite[Lemma 4.2 (b)]{Lubinsky_3}. For all $ z \in \mathbb{C} $, we have
\begin{align*}
	K_n(z, \overline{z}) = \int |K_n(z, x)|^2 d\mu(x) \geq \int_{x_0 - r/n}^{x_0 + r/n} |K_n(z,x)|^2 w(x) |x - x_0|^\alpha dx
\end{align*}
for large $ n $. After the substitution $ z = x_0 + a^*/n $, $ x = x_0 + t^*/n $, we have
\begin{align*}
	K_n& \Big( x_0 + \frac{a^*}{n}, x_0 + \frac{\overline{a^*}}{n} \Big) \\ & \geq \frac{1}{(\pi \omega_K(x_0))^{\alpha + 1}n^{\alpha + 1}}\int_{-r}^{r} \Big| K_n\Big( x_0 + \frac{a^*}{n}, x_0 + \frac{t^*}{n} \Big) \Big|^2 w(x_0 + t^*/n) |t|^\alpha dt
\end{align*}
which gives
\begin{align*}
	(\pi \omega_K(x_0))^{\alpha + 1}\geq \int_{-r}^{r} \frac{|f_n(a,t)|^2}{f_n(a,\overline{a})} \frac{K_n(x_0, x_0)}{n^{\alpha + 1}} w(x_0 + t/n) |t|^\alpha dt.
\end{align*}
By letting $ n \to \infty $ through the subsequence $ n_k $, (\ref{general_bulk}) gives
\[
	(\pi \omega_K(x_0))^{\alpha + 1} \geq \int_{-r}^{r} \frac{|f(a, t)|^2}{f(a, \overline{a})} (\pi \omega_K(x_0))^{\alpha + 1} \mathbb{L}_{\alpha}^{*}(0) |t|^{\alpha} dt,
\]
from which (\ref{normality_f_L2_bound}) follows. \\

(b) To prove that $ f(a, \cdot) $ belongs to the Cartwright class, we have to show that
\[
	\int_{-\infty}^{\infty} \frac{\log^+ |f(a,t)|}{1 + t^2} dt < \infty.
\]
Since $ f(a, \cdot) $ is entire, it is clear that $ \int_{-1}^{1} \frac{\log^+ |f(a,t)|}{1 + t^2} dt < \infty $. Next, using the well known facts $ \log^+ a b \leq \log^+ a + \log^+ b $ and $ \log^+ a^b = b \log^+ a $, we have
\[
	\int_{1}^{\infty} \frac{\log^+ |f(a,t)|}{1 + t^2} dt \leq C \bigg( \int_{1}^{\infty} \frac{\log^+ |f(a,t)|^2 |t|^{\alpha}}{1 + t^2} dt + \int_{1}^{\infty} \frac{\log^+ |t|}{1 + t^2} dt \bigg).
\]
The second integral is finite. For the first one, define
\[
	A_n = \{ t \in \mathbb{R}: e^{n} \leq |f(a,t)|^2 |t|^{\alpha} < e^{n + 1} \}.
\]
The bound (\ref{normality_f_L2_bound}) and Markov's inequality about the measure of level sets of $ L^1 $ functions gives that
\begin{align*}
	\int_{1}^{\infty} \frac{\log^+ |f(a,t)|^2 |t|^{\alpha}}{1 + t^2} dt & = \sum_{n=1}^{\infty} \int_{A_n} \frac{\log^+ |f(a,t)|^2 |t|^{\alpha}}{1 +t^2} dt \\
	& \leq \sum_{n=1}^{\infty} (n + 1)|A_n| \\
	& \leq C \sum_{n=1}^{\infty} (n + 1)  e^{-n} < \infty.
\end{align*}
The estimation of the integral $ \int_{-\infty}^{-1} \frac{\log^+ |f(a,t)|^2 |t|^{\alpha}}{1 + |t|^2} dt $ can be done in the same way, which shows that $ f(a, \cdot) $ belongs to the Cartwright class. \\

(c) This proof is identical to the one in \cite[Lemma 4.3]{Lubinsky_3}. Because $ f(a, \cdot) $ belongs to the Cartwright class, we have
\[
	\frac{\sigma_a}{\pi} = \lim_{r \to \infty} \frac{n(f(a, \cdot), r)}{2r},
\]
which, combined with (\ref{zeros_in_a_ball}), yields that $ \sigma_a $ is independent of $ a $.
\end{proof} 

From now on, since Lemma \ref{normality_f_lemma} (c) gives that $ \sigma_a $ is independent of $ a $, we shall denote the exponential type of $ f(a, \cdot) $ with $ \sigma $.

\begin{lemma}\label{lemma_f-L_inequality}
For all $ a \in \mathbb{R} $, we have
\begin{equation}\label{f-L_inequality}
\begin{aligned}
	\int_{-\infty}^{\infty} & \bigg( \frac{f(a/\sigma, t/\sigma)}{f(a/\sigma, a/\sigma)} - \frac{\mathbb{L}_{\alpha}^{*}(a, t)}{\mathbb{L}_{\alpha}^{*}(a, a)} \bigg)^2 |t|^\alpha dt \\
	& \phantom{aaaaaaa} \leq \frac{\sigma^{\alpha + 1}}{f(a/\sigma, a/\sigma)\mathbb{L}_{\alpha}^{*}(0,0)} - \frac{1}{\mathbb{L}_{\alpha}^{*}(a, a)}.
\end{aligned}
\end{equation}
Moreover,
\begin{equation}\label{exp_type_lower_estimate}
	\sigma \geq 1.
\end{equation}
\end{lemma}
\begin{proof}
(\ref{normality_f_L2_bound}) implies that $ |x|^{\alpha/2}f(a/\sigma, t/\sigma) \in \mathbb{L}^2(\mathbb{R}) $, therefore after expanding the left hand side of (\ref{f-L_inequality}), (\ref{reproducing_identity_equation}) and (\ref{reproducing_identity_L}) gives (\ref{f-L_inequality}). Using that the left hand side of (\ref{f-L_inequality}) is nonnegative, substituting $ a = 0 $ and keeping in mind that $ f(0,0) = 1 $ gives $ \sigma \geq 1 $.
\end{proof}

The inequality (\ref{f-L_inequality}) and (\ref{general_bulk}) imply that if $ \sigma = 1 $, then $ f(a, t) = \frac{\mathbb{L}_{\alpha}^{*}(a, t)}{\mathbb{L}_{\alpha}^{*}(0,0)} $ for all $ a, t \in \mathbb{R} $, which, since $ f(a, b) $ is entire in both variables, would imply Theorem \ref{main_theorem_universality_bulk}.

\begin{lemma}\label{MS_limit}
Let $ k > l $ be given integers. Then
\begin{equation}\label{MS_limit_final_inequality}
	\sum_{j=l+1}^{k-1} \frac{1}{f(\rho_j, \rho_j)} \leq \mathbb{L}_\alpha^*(0,0) \frac{\rho_{k}^{\alpha + 1} - \rho_{l}^{\alpha + 1}}{\alpha + 1} \leq \sum_{j=l}^{k} \frac{1}{f(\rho_j, \rho_j)}.
\end{equation}
\end{lemma}
\begin{proof}
The Markov-Stiejtles inequalities along with(\ref{Christoffel_polynomial_form}) imply, as in \cite[p. 33 (5.10)]{Freud}, that
\begin{equation}\label{MS_limit_inequality}
	\sum_{j=l+1}^{k-1} \frac{1}{K_n(t_{jn}(x_0), t_{jn}(x_0))} \leq \int_{t_{ln}(x_0)}^{t_{kn}(x_0)} d\mu(x) \leq \sum_{j=l}^{k} \frac{1}{K_n(t_{jn}(x_0), t_{jn}(x_0))},
\end{equation}
where $ t_{jn}(x_0) $ denotes the zeros of $ \psi_n(x_0, \cdot) = p_n(x_0) p_{n-1}(\cdot) - p_n(\cdot)p_{n-1}(x_0) $ centered around $ x_0 $ such that $ t_{0n}(x_0) = x_0 $. Now suppose that $ t_{ln}(x_0) $ and $ t_{kn}(x_0) $ belongs to $ (x_0 - \varepsilon_0, x_0 + \varepsilon_0) $, where $ \varepsilon_0 $ is so small such that $ \mu $ is absolutely continuous in this interval. Then, substituting $ x = x_0 + s^*/n $ (recall that $ s^* = s/(\pi \omega_K(x_0)) $ by definition) and using (\ref{zeros_of_K_n_as_zeros_of_psi_n}) with (\ref{zeros_of_f_n_as_zeros_of_K_n}), the integral in the middle takes the form
\begin{align*}
	\int_{t_{ln}(x_0)}^{t_{kn}(x_0)} d\mu(x) & = \int_{t_{ln}(x_0)}^{t_{kn}(x_0)} w(x) |x - x_0|^\alpha dx \\
	& = \frac{1}{n^{\alpha + 1}}\int_{\rho_{ln}(0)}^{\rho_{kn}(0)} \frac{w(x_0 + s^*/n)}{(\pi \omega_K(x_0))^{\alpha + 1}} |s|^\alpha ds.
\end{align*}
On the other hand, by definition $ \frac{K_n(t_{jn}(x_0), t_{jn}(x_0))}{K_n(x_0, x_0)} = f_n(\rho_{jn}(0), \rho_{jn}(0)) $. Multiplying with $ K_n(x_0, x_0) $ in (\ref{MS_limit_inequality}) we obtain
\begin{align*}
	\sum_{j=l+1}^{k-1} \frac{1}{f_n(\rho_{jn}(0), \rho_{jn}(0))} & \leq \frac{K_n(x_0, x_0)}{n^{\alpha + 1}} \int_{\rho_{ln}(0)}^{\rho_{kn}(0)} \frac{w(x_0 + s/n)}{(\pi \omega_K(x_0))^{\alpha + 1}} |s|^\alpha ds \\
	& \leq \sum_{j=l}^{k} \frac{1}{f_n(\rho_{jn}(0), \rho_{jn}(0))},
\end{align*}
which, after letting $ n $ to infinity and using (\ref{general_bulk}) with (\ref{zeros_converge}), yields (\ref{MS_limit_final_inequality}).
\end{proof}

The next lemma is an analogue of \cite[Lemma 5.3]{Lubinsky_3}, for which the proof also goes in an identical way.
\begin{lemma} Let $ \delta > 0 $ be arbitrary. \\
(a) There exists a positive integer $ L_+ $ such that if $ k > l > L_+ $ are selected in a way that
\begin{equation}\label{zero_spacing_ratio}
	\rho_k \leq (1 + \delta) \rho_l
\end{equation}
holds, we have
\begin{equation}\label{rho_spacing_1}
	k - l - 1 \leq (1 + \delta)^{|\alpha| + 1} \frac{\rho_k - \rho_l}{\pi}.
\end{equation}
Similarly, there exists a negative integer $ L_- $ such that if $ L_- > l > k $ are selected in a way that
\[
	|\rho_k| \leq (1 + \delta) |\rho_l|
\]
holds, 
\[
	l - k - 1 \leq (1 + \delta)^{|\alpha| + 1} \frac{|\rho_k| - |\rho_l|}{\pi}
\]
follows. \\
(b) For the function $ f(a, \cdot) $,
\[
	\limsup_{r \to \infty} \frac{n(f(a, \cdot), r)}{2r} \leq \frac{1}{\pi}
\]
holds. In particular, we have
\begin{equation}\label{exp_type_upper_estimate}
	\sigma \leq 1
\end{equation}
\end{lemma}
\begin{proof}
(a) We only show the existence of $ L_+ $, the existence of $ L_- $ follows similarly. (Or by reflecting the measure $ \mu $ around $ x_0 $.) Since (\ref{general_bulk}) gives that $ f(a,a) = \mathbb{L}_{\alpha}^{*}(a,a)/\mathbb{L}_{\alpha}^{*}(0,0) $, the number $ k - l - 1 $ can be written as
\[
	k - l - 1 = \sum_{j=l + 1}^{k - 1} \frac{\mathbb{L}_{\alpha}^{*}(\rho_j, \rho_j)}{\mathbb{L}_{\alpha}^{*}(0,0) f(\rho_j, \rho_j)}.
\]
If $ L $ is large enough, then (\ref{L_alpha_large_argument}) implies that for all $ j \geq L $,
\[
	\mathbb{L}_{\alpha}^{*}(\rho_j, \rho_j) \leq \frac{1 + \delta}{\pi \rho_j^\alpha}
\]
holds. Combining these with (\ref{MS_limit_final_inequality}), we obtain
\begin{align*}
	k - l - 1 & = \sum_{j=l + 1}^{k - 1} \frac{\mathbb{L}_{\alpha}^{*}(\rho_j, \rho_j)}{\mathbb{L}_{\alpha}^{*}(0,0) f(\rho_j, \rho_j)} \\
	& \leq \frac{1 + \delta}{\pi \mathbb{L}_{\alpha}^{*}(0,0)} \frac{1}{\min \{ \rho_l^{\alpha}, \rho_k^{\alpha}\}} \sum_{j= l + 1}^{k - 1} \frac{1}{f(\rho_j, \rho_j)} \\
	& \leq	\frac{1 + \delta}{\pi(1 + \alpha)}\frac{\rho_k^{\alpha + 1} - \rho_l^{\alpha + 1}}{\min \{ \rho_l^{\alpha}, \rho_k^{\alpha}\}} \\
	& \leq (1 + \delta)\frac{\rho_k - \rho_l}{\pi} \frac{\max \{ \rho_l^{\alpha}, \rho_k^{\alpha}\}}{\min \{ \rho_l^{\alpha}, \rho_k^{\alpha}\}},
\end{align*}
where the mean value theorem was used in the last step. (\ref{zero_spacing_ratio}) gives that
\[
	\frac{\max \{ \rho_l^{\alpha}, \rho_k^{\alpha}\}}{\min \{ \rho_l^{\alpha}, \rho_k^{\alpha}\}} \leq (1 + \delta)^{|\alpha|},
\]
therefore overall we have
\[
	k - l - 1 \leq (1 + \delta)^{|\alpha| + 1} \frac{\rho_k - \rho_l}{\pi},
\]
which gives (\ref{rho_spacing_1}).
The proof of (b) goes exactly as \cite[Lemma 5.3 (b)]{Lubinsky_3}.
\end{proof}
\textbf{Proof of Theorem \ref{main_theorem_universality_bulk}}. (\ref{exp_type_lower_estimate}) and (\ref{exp_type_upper_estimate}) gives that the exponential type of $ f(a, \cdot) $ is $ \sigma = 1 $. Substituting this back to the inequality (\ref{f-L_inequality}), we obtain that for all real $ b $, we have
\[
	f(a,b) = \frac{\mathbb{L}_{\alpha}^{*}(a,b)}{\mathbb{L}_{\alpha}^{*}(0,0)}, \quad a, b \in \mathbb{R}.
\]
Since $ f(a,b) $ is entire in both variables, it follows that the above equality holds for complex $ a, b $. Because the family $ \{ f_n(a,b) \}_{n=1}^{\infty} $ is normal and the above inequality is independent of the particular subsequence (recall that $ f(a,b) = \lim_{k \to \infty} f_{n_k}(a,b) $ for some subsequence $ n_k $), it follows that $ \lim_{n \to \infty} f_n(a,b) $ exists and it is $ f(a,b) $. Moreover, the convergence is uniform for $ a, b $ in compact subsets of the complex plane, as stated. \begin{flushright}$ \Box $ \end{flushright}

\end{document}